\documentclass[11pt,oneside,a4paper]{elsarticle}
\usepackage[a4paper,top=3cm,bottom=2cm,left=3cm,right=3cm,marginparwidth=1.75cm]{geometry}
 \usepackage{amsmath}\usepackage{etex}
\usepackage{amssymb}
\usepackage{amsfonts,amsthm,cases,mathtools}
\usepackage{threeparttable}
\usepackage{booktabs}
\usepackage{multirow}
\usepackage{stmaryrd}
\usepackage{chemarrow}
\usepackage[norelsize]{algorithm2e}
\usepackage{enumerate}
\usepackage{graphicx}
\usepackage[all]{xy}
\usepackage{tikz}
\usetikzlibrary{arrows}
\usepackage{multirow}
\usepackage{caption}
\usepackage[colorinlistoftodos]{todonotes}
\usepackage[colorlinks=true, allcolors=blue]{hyperref}

\usepackage{bm}
\usepackage{mathrsfs}
\newtheorem{theorem}{Theorem}[section]
\newtheorem{lemma}[theorem]{Lemma}

\newtheorem{example}[theorem]{Example}

\newtheorem{remark}[theorem]{Remark}

\numberwithin{figure}{section}
\numberwithin{equation}{section}
\newcommand{\mbb}{\mathbb }
\newcommand{\mrm}{\mathrm}

\newcommand{\mcal}{\mathcal}

\renewcommand{\div}{\operatorname{div}}

\begin{document}

%
%\maketitle

%\tableofcontents
\begin{frontmatter}
  \title{  A robust $C^0$ interior penalty method for a gradient-elastic\\ Kirchhoff plate model}
  		\author[zjnu]{Mingqing Chen}
\address[zjnu]{School of Mathematical Sciences, Zhejiang Normal University, Zhejiang 321004, China}
  \ead{chenmq@zjnu.edu.cn}

    \author[sjtu]{Jianguo Huang\fnref{hjgfootnote}}
\address[sjtu]{School of Mathematical Sciences, and MOE-LSC, Shanghai Jiao Tong University, Shanghai 200240, China}
	\ead{jghuang@sjtu.edu.cn}
	%\cortext[mycorrespondingauthor]{Corresponding author}
	\fntext[hjgfootnote]{Corresponding author. The work of this author was partially supported by NSFC (Grant No.\ 12071289) and the National Key Research and Development Project (2020YFA0709800).}
	
	\author[sufe]{Xuehai Huang\fnref{xuehaifootnote}}
    \address[sufe]{School of Mathematics, Shanghai University of Finance and Economics, Shanghai 200433, China} \ead{huang.xuehai@sufe.edu.cn}
	\fntext[xuehaifootnote]{The work of this author was partially supported by NSFC (Grant No.\ 12171300 and 12071289).}
	\begin{abstract}
	   This paper is devoted to proposing and analyzing a robust $C^0$ interior penalty method for a gradient-elastic Kirchhoff plate (GEKP) model over a convex polygon. The numerical method is obtained by combining the triangular Hermite element and a $C^0$ interior penalty method, which can avoid the use of higher order shape functions or macroelements. Next, a robust regularity estimate is established for the GEKP model based on our earlier result for a triharmonic equation on a convex polygon. Furthermore, some local lower bound estimates of the a posteriori error analysis are established. These together with an enriching operator and its error estimates lead to a C\'{e}a-like lemma. Thereby, the optimal error estimates are achieved, which are also robust with respect to the small size parameter. In addition, it is proved that this numerical method is convergent without any additional regularity assumption for the exact solution. Some numerical experiments are performed to verify the theoretical findings.
	\end{abstract}
	\begin{keyword}
		gradient-elastic Kirchhoff plate model; regularity; $C^0$ interior penalty method; robustness; error estimates;
	\end{keyword}	
\end{frontmatter}

 \section{Introduction}
% Let $\Omega\subset \mathbb R^2$ be a bounded domain with Lipschitz boundary.

To account for small-scale phenomena in micro-scale or even nano-scale structures, higher-order continuum theories were developed systematically by many researchers. The key feature is that the corresponding constitutive relations require to depend on additional size parameters to describe micro/nano scale properties of structures. We refer the reader to the references \cite{Cosserat1909,Toupin1962Elastic,MindlinTiersten1962Effects,Koiter1964couplestress,
	Mindlin1964Micro,Aifantis1984microstructural,AltanAifantis1992,RuAifantis1993simple} for details along this line. Most of the underlying mathematical models involve two or more size parameters. Aifantis and his collaborators proposed in \cite{AltanAifantis1992,RuAifantis1993simple} a one-parameter simplified strain gradient elasticity (SGE) theory for material deformation at the micro/nano scale, which is effective in simulating the deformation of elastic plastic materials. Its constitutive relation reads  (cf. \cite{NiiranenNiemi2017GEK})
 \begin{equation}\label{constitutive}
  \tilde{  \sigma}_{ij} = C_{ijkl}(  \varepsilon_{kl}- \iota^2  \varepsilon_{kl,mm} ),
\end{equation}
where, $\bm \varepsilon=(\varepsilon_{ij})$ is the strain tensor in elasticity, $\tilde{\bm \sigma}=(\tilde{ \sigma}_{ij})$ is the stress tensor of the SGE model, $C_{ijkl}$ is a fourth-order tensor and $\iota $ denotes the size parameter describing the length scale of the microstructure of the material under discussion. Here and below, we use Einstein's summation convention. Besides, for linear elastic problem,
\begin{equation}\label{Cilkj}
  C_{ijkl} = \lambda \delta_{ij}\delta_{kl} + \mu \delta_{ik}\delta_{jl} +\mu\delta_{il}\delta_{jk} ,
\end{equation}
where $\delta$ is the Kronecker delta function, and $\lambda$ and $\mu$ are the Lam\'e constants.
%Throughout the paper, we use the summation convention whereby summation is implied when an index is repeated exactly two times.

Given a thin plate occupying  a three dimensional region $\mathcal P=\Omega\times (-\frac t2 ,\frac t2)$, denote the displacement by $\bm u=(u_i) :\mathcal P\to \mathbb R^3$, where $\Omega$ is the midsurface of the plate and $t\ll{\rm diam}(\Omega)$ is the thickness of the plate. Under the Kirchhoff-Love hypothesis, the displacement field can be approximated as
 \begin{equation*}
    u_1= -x_3\frac{\partial w(x_1,x_2)}{\partial x_1},\quad     u_2 = -x_3\frac{\partial w(x_1,x_2)}{\partial x_2},\quad     u_3 =  w(x_1,x_2),
  \end{equation*}
 where $w: \Omega \to \mbb R$ is the transverse deflection at the midsurface $\Omega\times\{0\}$. Then, the classical strains in \eqref{constitutive} can be simplified to the plane strains $\varepsilon_{ij} = -x_3 \frac{\partial^2 w}{\partial x_i\partial x_j} $ with $i,j =1,2$. Next, we introduce a curvature tensor $\bm \kappa = \bm \kappa(\nabla w)=(\frac{\partial^2 w}{\partial x_i\partial x_j} ):\Omega\to \mathbb R^{2\times 2}$ as a restriction of the classic strain tensor field to the midsurface. Then, we can obtain the classic bending moment tensor $ \bm M = \mcal D((1-\nu)\bm\kappa + \nu \mrm{ tr }(\bm\kappa) \mathbf I )$ and the classical transverse force vector $\bm Q= \div\bm M$. Here, $\mathcal{D} = \frac{Et^3}{12(1-\nu^2) }$ is the bending rigidity related to Possion's ratio $ \nu \in (0,\,1/2)$ and Young's module $ E$. Then, the combination of the Kirchhoff plate bending equation, \eqref{constitutive} and \eqref{Cilkj}, one can deduce a sixth-order governing equation for the gradient-elastic Kirchhoff plate (GEKP) model
  \begin{equation}\label{IntroGEK}
  \div((1-\iota^2\Delta)\bm Q)=\mathcal{D}(\Delta^2 w - \iota^2 \Delta^3 w) = f,
  \end{equation}
 where $\iota\in (0,1)$ denotes the size parameter, and $f\in L^2(\Omega)$ is the transversal load. We refer to \cite{NiiranenNiemi2017GEK} for details of the above model.

 It is easy to observe that the mathematical model \eqref{IntroGEK} is a singularly perturbed problem with respect to the size parameter $\iota$. In fact, if  $\iota=0$, then the original problem will degenerate to a fourth-order elliptic problem. This implies that the solution of the reduced problem may not meet all of the boundary conditions of problem \eqref{model:GEK} and this would lead to the phenomenon of boundary layer as $\iota$ goes to zero. Thus, it is extremely necessary to develop a robust numerical method for approximating the GEKP model with respect to $\iota$.

As far as we know, there are some works on numerically solving \eqref{model:GEK}, such as the analytical methods \cite{NiiranenNiemi2017GEK,PapargyriGiannakopoulos2010,XuDengMeng2014,FuZhou2020,ZHOU2023}, the finite element methods (FEM) \cite{PegiosPapargyri2015,BabuPatel2019}, the extended Kantorovich method \cite{AshooriMahmoodi2013,WangHuangZhao2016}, the differential quadrature method \cite{MousaviNiiranen2015,IshaquddiGopalakrishnan2020,IshaquddinGopalakrishnan2021,LiJi2021}, the machine learning approach \cite{YanVescovini2023}, the isogeometric analysis method \cite{NiiranenKiendl2017GEK}, etc. However, most of these works mainly pay attention to numerical simulation using the underlying methods. In other words, there are few works on the construction of a robust numerical method for the GEKP model and developing the corresponding robust error analysis.
%,FranzRoosWachtel2014

Hence, one of the main motivations of our study is to develop a robust numerical method for solving the GEKP model \eqref{model:GEK}. %To avoid higher order polynomials or macroelements, making use of nonconforming finite elements is a common way to approximate higher-order equations.
We can refer to \cite{NilssenTaiWinther2001,WangHuangTang2018singular,ChenHuangHuang2023,BrennerNeilan2011} and references therein for details about some robust nonconforming FEMs for approximating singularly perturbed fourth-order problems. However, these numerical methods only possess the sharp half-order convergence rate due to the effect of boundary layer. Inspired by the Nitsche's technique \cite{Nitsche1971} and the penalty technique \cite{Arnold1982,BrennerSung2005IPFourth}, the authors in \cite{Guzman2012Nitsche,HuangShiWang2021singular} devised some robust nonconforming finite element methods by imposing one of boundary conditions weakly for a singularly perturbed fourth-order problem. See also the work in \cite{WangHuangTang2018singular}.

 %cut finite element method \cite{BurmanHansboLarson2022}, the finite difference method \cite{Long2012,Mohanty2010,Serban1965}, nonconforming method, the isogeometric analysis method \cite{TagliabueDedeQuarteroni2014}, multi¨Clevel method \cite{KarageorghisChen2025}
Furthermore, there has developed a few of FEMs that approximate triharmonic equations, including the conforming FEM \cite{HuLinWu2023,ChenHuang2021Decomposition}, the decoupled mixed FEM \cite{Schedensack2016,AnHuangZhang2024,Gallistl2017}, the nonconforming FEM \cite{WuXu2019,HuZhang2019,HuZhang2017,WangXu2013,LiWu2024}, the $C^0$ interior penalty method \cite{Gudi2011sixthIPDG,ChenLiQiu2022}, the virtual element method \cite{DassiMoraReales2024,ChenHuangWei2022,AntoniettiManzini2020,ChenHuang2020,Huang2020Nonconforming}, and hybrid high-order method \cite{RenLiu2024}. It is well-known that the conforming finite element methods for approaching high-order partial differential equations are hard to be constructed and implemented due to the quite a number of degrees of freedom (DoFs) and the existence of supersmooth DoFs. Therefore we prefer to apply less smooth finite element. In particular, the authors in \cite{Gudi2011sixthIPDG,ChenLiQiu2022} proposed some $C^0$ interior penalty methods for a third-order Laplace equation using the cubic Lagrange finite element space. This element only has 10 DoFs. Making use of the a posteriori analysis techniques, some C\'{e}a-like lemmas have been established with minimal regularity assumptions on the exact solution. This kind of methods has an advantage of using a small number of degrees of freedom. Thus, we are tempted to develop an interior penalty method to approximate such singularly perturbed high-order problems.

%The convergence results would depend only on the regularity of the solution of the reduced problem as $\iota$ goes to zero. %As a matter of fact, the optimal convergence rate can be observed from numerical verification.

 % Then, a quasi-optimal error estimate is deduced.

Another purpose of our study is to build up the regularity theory for the GEKP problem \eqref{model:GEK} and further to deduce robust error analysis of the underlying numerical method. In \cite{NilssenTaiWinther2001}, the authors derived a uniform regularity estimate for a singularly perturbed fourth-order problem with the help of the regularity result for a biharmonic equation. Recall that the regularity theory of a triharmonic equation on a convex polygon has been established in \cite{ChenHuangHuang2023}. This together with the regularity theory of the reduced problem \eqref{model:KPB} in \cite{Grisvard1985} can lead to the following robust regularity estimate
\[
  |w-w_0|_2 + \iota\|w\|_3 + \iota^2\|w\|_4\lesssim \iota^{1/2}\|f\|_0.
  \]

The main results of our work in this paper can be summarized as follows. Borrowing the ideas in  \cite{Nitsche1971,Arnold1982,BrennerSung2005IPFourth,Gudi2011sixthIPDG}, we design a $C^0$ interior penalty method to discretize problem \eqref{model:GEK}. The element we used here is the triangular Hermite element having 10 DoFs, which is easy to implement numerically. Using the similar approach as in \cite{Gudi2011sixthIPDG,Gudi2010medius,Gudi2010DG,HuangShiWang2021singular}, %
we first introduce a quasi interpolation operator and an enriching operator connecting the Hermite element space to a $C^2$ element space. Next, we derive some local lower bound estimates of a posteriori error analysis for the resulting numerical method, which will be exploited to achieve a C\'{e}a-like lemma. Thereby, the optimal and robust error estimates with respect to the size parameter can be obtained. Furthermore, it can be shown that the numerical method we proposed is convergent for any $w\in H^3_0(\Omega)$.

The rest of this paper is arranged as follows. In Sect. \ref{SectNotation}, we introduce some notation and existing results for further analysis.
In Sect. \ref{SectGEK}, we present the variational formulation of problem \eqref{model:GEK} and develop the underlying regularity estimate relying on Theorem 3.5 in \cite{ChenHuangHuang2023}. In Sect. \ref{SectFEM}, we design a $C^0$ interior penalty method of problem \eqref{model:GEK} and prove its well-posedness. In Sect. \ref{Sect:InterpolationError}, we define two interpolation operators and derive their error estimates. In Sect. \ref{Sect:ErrorAnalysis}, a C\'{e}a-like lemma can be derived with the aid of some interpolation error estimates and the a posteriori analysis techniques. Then, we are able to analyse the convergence properties of the method. In the last section, the numerical experiments are performed to verify the theoretical findings.

\section{Preliminaries}\label{SectNotation}

   Throughout this paper, let $\Omega$ be a bounded convex domain, which is partitioned into a family of shape regular triangles $\mathcal{T}_h=\{K\}$. That is to say that there exists a constant $\gamma_0\geq 1$ such that
 \[
  \frac{h_K}{\rho_K}\leq \gamma_0 \quad \forall \;K\in \mcal T_h,
 \]
 where $h_K={\rm diam}(K)$ and $h=\max_{K\in \mathcal{T}_h}h_K$. Use $\mathcal V(\mathcal T_h)$ and $\mathcal V^i(\mathcal T_h)$ (resp. $\mathcal{E}(\mathcal T_h)$ and $\mathcal{E}^i(\mathcal T_h)$) to stand for the sets of all and interior vertices (resp. edges) of $\mathcal T_h$, respectively.
  % Analogously, use $\mathcal V(K)$ and $\mathcal{E}(K)$ to symbolize the sets of all vertices and edges of $K$, respectively.
   We then introduce the patch of vertices and edges for further use.
   For any vertex $\delta\in \mathcal V(\mathcal T_h)$ and edge $e\in \mathcal{E}(\mathcal T_h)$, write $\omega_{\delta}$ and $\omega_{e}$ to be respectively the unions of all triangles in $\mathcal{T}_{\delta}$ and $\mathcal{T}_{e}$, where $\mathcal{T}_{\delta}$ and $\mathcal{T}_{e}$ is the set of all triangles in $\mathcal{T}_h$ sharing common vertex $\delta$ and edge $e$, respectively.
   %For any $ \delta\in \mathcal V^{\partial}(\mathcal T_h):=\mathcal V(\mathcal T_h)\backslash\mathcal V^i(\mathcal T_h)$, denote by $\mathcal{T}_{\delta}^{\partial}$ all triangles $K$ in $\mathcal{T}_{\delta}$ such that $\mathcal E(K)\cap\mathcal E^{\partial}(\mathcal T_h)\not = \emptyset$.
   In addition, for a finite set $A$, denote by $\# A$ its cardinality. For each $K \in \mcal T_h$, denote by $\bm n_K$ the unit outward normal to $\partial K$ and denote the unit tangential direction of $\partial K$ by $\bm t_K$ obtained by rotating $\bm n_K$ $90^\circ$ counterclockwise. Without causing any confusion, abbreviate $\bm n_K$ and $\bm t_K$ as $\bm n$ and $\bm t$ for simplicity. For all $e \in \mathcal{E}^i(\mathcal T_h)$ sharing by two triangles $K^+$ and $K^-$, we preset the unit normal vector of $e$ by $\bm n_e\coloneqq \bm n_+$, where $\bm n_+$ is the unit outward normal to $e$ of $K^+$ and denote the unit tangential direction of $e$ by $\bm t_e$ obtained by rotating $\bm n_e$ $90^\circ$ counterclockwise. Define the jump and average on $e$ as follows:
 \begin{align*}
  [\![\partial_{\bm{a}} v]\!] |_e \coloneqq   \partial_{\bm a_e}v|_{K^+} -   \partial_{\bm a_e}v|_{K^-},
     \quad  &
      \{\!\{\partial_{\bm{a}}  v\}\!\}  |_e \coloneqq \frac{  \partial_{\bm a_e}v|_{K^+} + \partial_{\bm a_e }v|_{K^-}} 2,
       \\
        [\![\partial_{\bm{ab}} v]\!] |_e \coloneqq   \partial_{\bm a_e\bm b_e}v|_{K^+} -   \partial_{\bm a_e\bm b_e}v|_{K^-},
     \quad  &
    \{\!\{\partial_{\bm{ab}}  v\}\!\}  |_e \coloneqq \frac{  \partial_{\bm a_e\bm b_e}v|_{K^+} + \partial_{\bm a_e\bm b_e}v|_{K^-}} 2,
     \\
     [\![\partial_{\bm{abc}} v]\!] |_e \coloneqq   \partial_{\bm a_e\bm b_e\bm c_e}v|_{K^+} -   \partial_{\bm a_e\bm b_e\bm c_e}v|_{K^-},
     \quad & \{\!\{\partial_{\bm{abc}}  v\}\!\}  |_e \coloneqq \frac{  \partial_{\bm a_e\bm b_e\bm c_e}v|_{K^+} + \partial_{\bm n_e\bm n_e\bm t_e}v|_{K^-}} 2.
   \end{align*}
   where $\bm a,\bm b,\bm c= \bm n,\bm t$.
  % \begin{align*}
%     [\![\partial_{\bm{n}} v]\!] |_e \coloneqq   \partial_{\bm n_e}v|_{K^+} -   \partial_{\bm n_e}v|_{K^-},
%     \quad  &
%       \{\!\{\partial_{\bm{n}}  v\}\!\}  |_e \coloneqq \frac{  \partial_{\bm n_e}v|_{K^+} + \partial_{\bm n_e }v|_{K^-}} 2,
%       \\
%        [\![\partial_{\bm{nn}} v]\!] |_e \coloneqq   \partial_{\bm n_e\bm n_e}v|_{K^+} -   \partial_{\bm n_e\bm n_e}v|_{K^-},
%     \quad  &
%    \{\!\{\partial_{\bm{nn}}  v\}\!\}  |_e \coloneqq \frac{  \partial_{\bm n_e\bm n_e}v|_{K^+} + \partial_{\bm n_e\bm n_e}v|_{K^-}} 2,
%     \\
%       [\![\partial_{\bm{nt}} v]\!] |_e \coloneqq   \partial_{\bm n_e\bm t_e}v|_{K^+} -   \partial_{\bm n_e\bm t_e}v|_{K^-},
%     \quad&
%     \{\!\{\partial_{\bm{nnn}}  v\}\!\}  |_e \coloneqq \frac{  \partial_{\bm n_e\bm n_e\bm n_e}v|_{K^+} + \partial_{\bm n_e\bm n_e\bm n_e}v|_{K^-}} 2,
%     \\
%    & \{\!\{\partial_{\bm{nnt}}  v\}\!\}  |_e \coloneqq \frac{  \partial_{\bm n_e\bm n_e\bm t_e}v|_{K^+} + \partial_{\bm n_e\bm n_e\bm t_e}v|_{K^-}} 2,
 % \\
%    & \{\!\{\partial_{\bm{tt}}  v\}\!\}  |_e \coloneqq \frac{  \partial_{\bm t_e \bm t_e}v|_{K^+} + \partial_{ \bm t_e\bm t_e}v|_{K^-}} 2.
%   \end{align*}
   For any $e \in \mcal E^{\partial }(\mcal T_h) $, the associated jump and average are given by
    \begin{align*}
     [\![\partial_{\bm{a}} v]\!] |_e \coloneqq   \partial_{\bm a_e}v ,
     \quad  &
       \{\!\{\partial_{\bm{a}}  v\}\!\}  |_e \coloneqq    \partial_{\bm a_e}v  ,
       \\
        [\![\partial_{\bm{ab}} v]\!] |_e \coloneqq   \partial_{\bm a_e\bm b_e}v ,
     \quad  &
    \{\!\{\partial_{\bm{ab}}  v\}\!\}  |_e \coloneqq  \partial_{\bm a_e\bm b_e}v  ,
     \\
     [\![\partial_{\bm{abc}} v]\!] |_e \coloneqq   \partial_{\bm a_e\bm b_e\bm c_e}v ,
     \quad & \{\!\{\partial_{\bm{abc}}  v\}\!\}  |_e \coloneqq \partial_{\bm a_e\bm b_e\bm c_e}v .
   \end{align*}
%   \begin{align*}
%     [\![\partial_{\bm{n}} v]\!] |_e \coloneqq   \partial_{\bm n_e}v
%     \quad  &
%       \{\!\{\partial_{\bm{nn}}  v\}\!\}  |_e \coloneqq   \partial_{\bm n_e\bm n_e}v
%       \\
%        [\![\partial_{\bm{nn}} v]\!] |_e \coloneqq   \partial_{\bm n_e\bm n_e}v
%     \quad  &
%     \{\!\{\partial_{\bm{nnn}}  v\}\!\}  |_e \coloneqq   \partial_{\bm n_e\bm n_e\bm n_e}v
%     \\
%       [\![\partial_{\bm{nt}} v]\!] |_e \coloneqq   \partial_{\bm n_e\bm t_e}v,
%     \quad&
%     \{\!\{\partial_{\bm{nnt}}  v\}\!\}  |_e \coloneqq   \partial_{\bm n_e\bm n_e\bm t_e}v.
%   \end{align*}
   Clearly,
   \[
[\![ uv ]\!] =  \{\!\{u\}\!\}  [\![ v ]\!] + [\![ u]\!] \{\!\{v\}\!\}  .
\]

%For any vertex $\delta\in \mathcal V(\mathcal T_h)$, edge $e\in \mathcal{E}(\mathcal T_h)$ and $K\in \mathcal T_h$, write $\omega_{\delta}$, $\omega_{e}$ and $\omega_K$ to be respectively the unions of all triangles in $\mathcal{T}_{\delta}$, $\mathcal{T}_{e}$ and $\mathcal T_K$, where $\mathcal{T}_{\delta}$ and $\mathcal{T}_{e}$ is the set of all triangles in $\mathcal{T}_h$ sharing common vertex $\delta$ and edge $e$, respectively, $\mathcal T_K$ is the set of all triangles in $\mathcal{T}_h$ intersecting with $K$.

   Given a bounded domain $D$ and any integers $m,k\geq0$, denote by $H^m(D)$ the standard Sobolev spaces on $D$ with norm $\|\cdot\|_{m,D}$ and semi-norm $|\cdot|_{m,D}$, and $H_0^m(D)$ the closure of $C_0^{\infty}(D)$ with respect to $\|\cdot\|_{m,D}$. The notation $(\cdot,\,\cdot)_{D}$ symbolizes the standard inner product on $D$. For $D=\Omega$, we abbreviate $\|\cdot\|_{m,\Omega}$, $|\cdot|_{m,\Omega}$ and $(\cdot,\,\cdot)_{D}$ as $\|\cdot\|_{m}$, $|\cdot|_{m}$ and $(\cdot,\,\cdot)$, respectively.
   Let $\mathbb{P}_{k}(D)$ be the set of all polynomials on $D$ with the total degree up to $k$. Denote by $\Pi_D^k$ the standard $L^2$ projection operator from $L^2(D)$ to $\mathbb P_k(D)$.

   For simplicity of expression, we use $\lesssim$ to stand for $\leq C$, where $C$ may be a generic positive constant independent of the mesh size $h$ and the material parameter $\iota$. And $a\eqsim b$ indicates $a\lesssim b\lesssim a$.
    Here $\nabla_h$ (resp. $\Delta_h$) is the elementwise version of $\nabla$ (resp. $\Delta$) with respect to $\mathcal T_h$.

   For forthcoming theoretical analysis, recall some basic results. As given in \cite[Theorem 1.5.1.10]{Grisvard1985}, the following multiplicative trace inequality holds
        \begin{align*}
		\|v\|_{0,\partial{D}}^2 & \lesssim \|v\|_{0,D} \|v\|_{1,D} \quad \forall\;v\in H^1(D),
		%\label{Trace_Iq_H2}
		\end{align*}
where $D$ is a bounded domain with Lipschitz boundary. And by the scaling argument, there exists a positive constant $C(\varepsilon)$ for any $\varepsilon> 0$,  such that
		\begin{align}\label{eq:trace1}
		\|v\|_{0,\partial{K}} & \lesssim \varepsilon h_K^{1/2} |v|_{1,K}
		+ C(\varepsilon) h_K^{-1/2} \|v\|_{0,K} \quad \forall\; v\in H^1(K),\,K\in \mcal T_h.
		%\label{Trace_Iq_H1}
		\end{align}
Moreover, by the inverse inequality of polynomials, one has readily
\begin{align}
\|\nabla^m p \|_{0,\partial K}  \lesssim &  h_K^{-1/2}\|\nabla^m p \|_{0,K} \quad \forall\;p\in \mbb P_k(K),\,K\in \mcal T_h.
   \label{eq:GEKStranginverseP0}
\end{align}

\section{Regularity results}\label{SectGEK}

     Let us consider the following boundary value problem with the Dirichlet boundary conditions:
      \begin{equation}\label{model:GEK}
        \begin{cases}
       \mathcal{D}(  \Delta^2w - \iota^2 \Delta^3 w) = f & \textrm{in}~\Omega ,\\
          w = \partial_{\bm{n}} w = \partial_{\bm{nn}} w = 0 &\textrm{on}~\partial\Omega.
        \end{cases}
      \end{equation}
      Without loss of generality, set ${\mathcal D}=1$ throughout this paper. The weak formulation of problem \eqref{model:GEK} is to find $w\in V\coloneqq H^3_0(\Omega)$ such that
     \begin{equation}\label{weakForm:GEK}
       \iota^2a(w,v) + b(w,v) = (f,v)\quad \forall\;v\in V,
     \end{equation}
     where
      \[
  a(w,v) =   (\nabla^3 w, \nabla^3 v) \quad\text{and}\quad  b(w ,v) =  (\nabla^2   w , \nabla^2 v).% \quad \forall\; w,v\in H^3_0(\Omega).
  \]
Set $\iota=0$, problem \eqref{model:GEK} will reduce to a fourth-order elliptic problem:
  \begin{equation}\label{model:KPB}
        \begin{cases}
           \Delta^2 w_0  = f  & \textrm{in}~\Omega ,\\
          w_0 = \partial_{\bm{n}} w_0 = 0&\textrm{on}~\partial\Omega .
        \end{cases}
      \end{equation}
     Its weak formulation is to find $w_0\in  H^2_0(\Omega)$ such that
     \begin{equation}\label{weakForm:KPB}
        b(w_0 ,v) = (f,v)\quad \forall\;v\in H^2_0(\Omega).
     \end{equation}
   It's worth noting that for any convex polygon $\Omega$, the operator $\Delta^2$ is an isomorphism from $H^3(\Omega)\cap H^2_0(\Omega)$ to $H^{-1}(\Omega)$ (cf. \cite[Corollary~7.3.2.5]{Grisvard1985}). Then, $w_0\in H^3(\Omega)$ and there exists a positive constant $C$ such that
   \begin{equation}\label{eq:regularityKPB}
     \|w_0\|_3\leq C\|f\|_0.
   \end{equation}

For the regularity estimate of \eqref{model:GEK} , we recall a useful result of the triharmonic equation.
\begin{theorem}(\cite[Theorem 3.5]{ChenHuangHuang2023})\label{th:TriharmonicRegularity}
	Let $\Omega$ be a convex polygon. For any $w\in H^3_0(\Omega)$ satisfying
	\begin{equation*} %\label{eq:Triharmonic}
	-\Delta^3 w =f \in H^{-2}(\Omega),
	\end{equation*}
  it holds $w\in H^4(\Omega)$ and
	\begin{equation*}%\label{eq:sixthorderregularity}
	\| w\|_{4} \lesssim \|f\|_{-2}.
	\end{equation*}
\end{theorem}
Following the argument in \cite[Lemma 5.1]{NilssenTaiWinther2001}, we have the following regularity result.
\begin{theorem}\label{th:GEKRegularity}
Assume $\Omega$ be a convex polygon. Let $w\in H^3_0(\Omega)$ and $w_0\in H^2_0(\Omega)$ be the solutions of problems \eqref{weakForm:GEK} and \eqref{weakForm:KPB}, respectively. Then it holds $w\in H^4(\Omega)$ and
  \begin{equation}\label{eq:GEKRegularity}
    |w-w_0|_2 + \iota\|w\|_3 + \iota^2\|w\|_4\lesssim \iota^{1/2}\|f\|_0.
  \end{equation}
  Here, the hidden constant is independent of the size parameter $\iota$.
\end{theorem}
\begin{proof}
  The first equation of problem \eqref{model:GEK} can be reformulated as
  \[
    - \Delta^3 w = \iota^{-2}\Delta^2(w_0  -  w)\in H^{-2}(\Omega),
  \]
  which along with Theorem \ref{th:TriharmonicRegularity} yields $w\in H^4(\Omega)$ and
  \begin{equation}\label{eq:GEKRegularity1}
    \|w\|_4\lesssim \|\iota^{-2}\Delta^2(w_0  -  w)\|_{-2}\lesssim  \iota^{-2} | w_0  -  w  |_2.
  \end{equation}
 Thus, it suffices to prove
   \begin{equation*}%\label{eq:GEKRegularity8}
    |w-w_0|_2 + \iota\|w\|_3 \lesssim \iota^{1/2}\|f\|_0.
  \end{equation*}
  For any $v\in H^2_0(\Omega)\cap H^3(\Omega)$, it holds
  \[
   v|_{\partial\Omega} = \partial_{\bm{n}}v|_{\partial\Omega} = \partial_{\bm{nt}} v |_{\partial\Omega}= 0.
  \]
% Notice that $w\in H^3_0(\Omega)\cap H^4(\Omega)$ and ~, ÓÐ%%£¨Õâ¸öÐÔÖÊÊÇ·ñ·Åµ½Ç°Ãæ·Ö²¿»ý·ÖÖУ©
Then, by \eqref{weakForm:GEK} and \eqref{weakForm:KPB}, we have
  \begin{align}\label{eq:GEKRegularity2}
    \iota^2(\nabla^3w,\nabla^3 v) +(\nabla^2w,\nabla^2 v)-\iota^2 (\partial_{\bm{nnn}} w, \partial_{\bm{nn}} v)_{\partial \Omega} = (\nabla^2w_0,\nabla^2 v) .
  \end{align}
    Choosing $v$ as $w-w_0$ in \eqref{eq:GEKRegularity2}, we have by integration by parts that
    \begin{align}\label{eq:GEKRegularity3}
      \iota^2|w|_3^2  + |w-w_0|_2^2
       &    =   \iota^2 (\nabla^3w,\nabla^3 w_0)  - \iota^2(\partial_{\bm{nnn}} w, \partial_{\bm{nn}}  w_0 )_{\partial \Omega}.
      %        \notag\\
     %  &   \leq   \iota^2  C| w |_3\|f \|_0+  \iota^2 (\partial_{\bm{nnn}} w, \partial_{\bm{nn}}  w_0 )_{\partial \Omega}.
         %     \notag\\
%      &\leq  |w|_4 |f|_0 + \|\partial_{\bm{nnn}} w\|_{0,\partial \Omega}\|\partial_{\bm{nnn}} w\|_{0,\partial \Omega}
    \end{align}
    By the multiplicative trace inequality, \eqref{eq:regularityKPB} and \eqref{eq:GEKRegularity1},
    \begin{align*}%\label{eq:GEKRegularity4}
     -\iota^2 (\partial_{\bm{nnn}} w, \partial_{\bm{nn}} w_0 )_{\partial \Omega}
     & \lesssim \iota^2\|w\|_3^{1/2}\|w\|_4^{1/2}\|w_0\|_2^{1/2}\|w_0\|_3^{1/2}
               \notag\\
     & \lesssim \iota\|w\|_3^{1/2}|w-w_0|_2^{1/2} \|f\|_0.
    \end{align*}
  %  where $C_2$ is the hidden constant in the multiplicative trace inequality. %%  From \eqref{eq:GEKRegularity1} and the arithmetic geometric mean inequality we get
%%       \begin{align}\label{eq:GEKRegularity5}
%%    \delta \iota^3\|w\|_3\|w\|_4
%%   & \leq \frac{1}{2}(\iota^2\|w\|_3^2 +  \delta^{2}\iota^4\|w\|_4^2  )
%%       \notag\\
%%   & \leq  \frac{1}{2}(\iota^2\|w\|_3^2 +\delta^{2}C_1^2|w-w_0|_2^2  )
%%      \end{align}
      From \eqref{eq:regularityKPB} we get
      \begin{align*}%\label{eq:GEKRegularity6}
         \iota^2 (\nabla^3w,\nabla^3 w_0)  \leq \iota^2 |w|_3|w_0|_3  \lesssim \iota^2 |w|_3 \|f\|_0  .
      \end{align*}
    Combining \eqref{eq:GEKRegularity3} and the last two inequalities gives
   \begin{align*}
          \iota^2|w|_3^2  + |w-w_0|_2^2 \lesssim& (\iota^{1/2}\|w\|_3^{1/2}|w-w_0|_2^{1/2} + \iota^{3/2} |w|_3)\iota^{1/2} \|f\|_0
          \notag\\
          \lesssim & (\iota \|w\|_3 + |w-w_0|_2) \iota^{1/2}\|f\|_0,
  \end{align*}
  which along with the absorbing technique yields the desired result.
\end{proof}

 \section{$H^3$ nonconforming finite element method}\label{SectFEM}

For a triangle $K\in \mathcal T_h$, let $\mathcal V(K)$ and $\mathcal E(K)$ be the set of all vertices and edges of $K$, respectively. %Denote the barycenter of $K$ by $a_K= (a_1+a_2+a_3)/3$.
%For $i=1,2, 3$, denote by $\lambda_i$ the barycentric coordinate corresponding to $ a_i$ and there exists $e_i\in \mathcal E(K)$ such that $\lambda_i|_{e_i}=0$.Let $b_K = \lambda_1  \lambda_2  \lambda_3 =\lambda_jb_{e_j}$ be the bubble function of $K$, where $b_{e_j}$ is the bubble function of $e_j$ with $j=1,2, 3$.
%%Set $\mathcal V^i(K):=\mathcal V(K)\cap \mathcal V^i(\mathcal T_h)$.

  \subsection{Finite element space}

 The Hermite element in a triangle takes $ V(K):= \mathbb P_{3}(K )  $ as the local shape function space. The corresponding DoFs are
%  \begin{align} \label{GEKdof:H3ncfemdof}
%         \{  v( a_0),  v (a_i),    \nabla v( a_i)(a_j-a_i)  ,1\leq i\not= j\leq 3 \}.
%      \end{align}
     %\begin{align}
%           v( a_0), &  \label{GEKdof:H3ncfemdof1}\\
%      v (a_i)   &  ,
%      \label{GEKdof:H3ncfemdof2}\\
%      \nabla  v(a_i) &
%      \label{GEKdof:H3ncfemdof3}
%      \end{align}
     % where $a_0 = (a_1+a_2+a_3)/3$ is the barycenter of $K$, $1\leq i\leq 3 $.
       \begin{equation}\label{GEKdof:H3ncfemdof}
          \left\{
        \begin{aligned}
         & v(\zeta)\quad  \;\;\;\forall \; \zeta\in \mcal V(K),\\
        & \nabla  v(\zeta)\quad  \forall \; \zeta\in \mcal V(K), \\
         & v(\zeta_K).
        \end{aligned}
        \right.
      \end{equation}
 where $\zeta_K$ is the barycenter of $K$.
     Let
\begin{align*}
  V_h \coloneqq \{ &v\in H^1(\Omega):\, v|_K \in V(K) \text{ for each $K\in \mcal T_h$, $v$ and $\nabla v$ are continuous at all }\\
   &\quad\quad \quad  \quad\text{vertices of triangulation $\mathcal T_h$ and all DoFs \eqref{GEKdof:H3ncfemdof} vanish on boundary}\}.
\end{align*}
%\begin{remark}
%    Dofs $\nabla v( a_i)(a_j-a_i)  $ is equivalent to the value of partial derivative at vertices:
%   \[
%   h_{a_i}\nabla v( a_i) =h_{a_i} ( \partial_1 v( a_i),\partial_2v( a_i )),
%   \]
%   where $ h_{a_i}$ is the characteristic length.
%\end{remark}

     \subsection{Discrete formulation}
Assume $w$ is sufficiently smooth. For any $v\in V_h\subset H^1_0(\Omega)$, it follows from integration by parts that
      \begin{align}\label{eq:discreteFormIntegration1}
      -  (\Delta^3 w, v) & =  (\nabla  \Delta^2 w, \nabla v)
      \notag\\
      &= \sum_{K\in \mcal T_h} \Big(  - (\nabla^2 \Delta w, \nabla^2 v)_K  + (\partial_{\bm{n}} \nabla \Delta w,\nabla v)_{\partial K} \Big)
      \notag\\
      &= \sum_{K\in \mcal T_h} \Big(   (\nabla^3  w, \nabla^3 v) _K -  (\partial_{\bm{n}} \nabla^2 w,\nabla^2 v)_{\partial K}  + (\partial_{\bm{n}} \nabla  \Delta w,\nabla v)_{\partial K}\Big)
      \end{align}
  and
       \begin{align}\label{eq:discreteFormIntegration7}
        (\Delta^2 w, v)
      &= \sum_{K\in \mcal T_h} \Big(   (\nabla^2  w, \nabla^2 v)_K  -  (\partial_{\bm{n}} \nabla w,\nabla  v)_{\partial K}   \Big).
      \end{align}
     Note that
      \begin{align} \label{eq:discreteFormIntegration9}
        \partial_i v =  n_i\partial_{\bm{n}} v   + t_i \partial_{\bm t} v\quad \forall\; v\in H^1(K),\;i=1,2,
      \end{align}
      with any $\bm n=(n_1,n_2)^{\intercal} \in \mbb R^2$ and  $\bm t=(t_1,t_2)^{\intercal} \in \mbb R^2$.
      Then we get easily that
      \begin{align}
        \nabla^2 v &= \bm n \partial_{\bm{nt}}v \bm t^{\intercal} + \bm n \partial_{\bm{nn}}v \bm n^{\intercal} +\bm t \partial_{\bm{tn}} v\bm n^{\intercal}  +\bm t \partial_{\bm{tt}} v\bm t^{\intercal}, \label{eq:discreteFormIntegration2}
        \\
      \partial_{\bm{n}} \nabla \Delta w&= \partial_{\bm{nn}}\Delta w \bm n + \partial_{\bm{nt}}\Delta w \bm t,\label{eq:discreteFormIntegration3}
      \\
   \partial_{\bm{n}}\nabla^2 w  &= \bm n \partial_{\bm{nnn}}w \bm n^{\intercal} + \bm n \partial_{\bm{nnt}}w \bm t^{\intercal} +  \bm t \partial_{\bm{nnt}}w \bm n^{\intercal} + \bm t \partial_{\bm{ntt}}w \bm t^{\intercal} ,\label{eq:discreteFormIntegration4}\\
   \partial_{\bm{n}} \nabla w & =  \partial_{\bm{nn}}w \bm n^{\intercal} + \partial_{\bm{nt}}w \bm t^{\intercal} .\label{eq:discreteFormIntegration8}
      \end{align}
      Combining \eqref{eq:discreteFormIntegration9}-\eqref{eq:discreteFormIntegration4} and \eqref{eq:discreteFormIntegration1}, we find
      \begin{align}\label{eq:discreteFormIntegration5}
      -   (\Delta^3 w, v)
      = &  \sum_{K\in \mcal T_h}\Big(  (\nabla^3  w, \nabla^3 v)_K  + ( \partial_{\bm{nn}}\Delta w  ,\partial_{\bm{n}} v)_{\partial K}+ ( \partial_{\bm{nt}}\Delta w  ,\partial_{\bm t} v)_{\partial K}\Big)
      \notag\\
      & - \sum_{K\in \mcal T_h}\Big( (\partial_{\bm{nnn}} w, \partial_{\bm{nn}} v)_{\partial K}
          +2 (\partial_{\bm{nnt}} w, \partial_{\bm{nt}} v)_{\partial K}
         + (\partial_{\bm{ntt}} w, \partial_{\bm{tt}} v)_{\partial K} \Big) .
      \end{align}
      Due to the fact that $v\in H^1_0( \Omega )\cap C^0(\overline{\Omega})$ and $w\in H^3_0(\Omega)$, it holds
     \begin{align}\label{eq:discreteFormIntegration6}
      -  (\Delta^3 w, v)
      = &\sum_{K\in \mcal T_h} (\nabla^3  w, \nabla^3 v)_K
      \notag\\
      &+ \sum_{e\in \mcal E(\mcal T_h)}  (\{\!\{ \partial_{\bm{nn}}\Delta w\}\!\} ,[\![ \partial_{\bm{n}} v]\!] )_{e}
       + \sum_{e\in \mcal E(\mcal T_h)}  ([\![ \partial_{\bm{n}} w]\!],\{\!\{ \partial_{\bm{nn}}\Delta v\}\!\}  )_{e}
      \notag\\
      & - \sum_{e\in \mcal E(\mcal T_h)}     ( \{\!\{\partial_{\bm{nnn}} w\}\!\}, [\![ \partial_{\bm{nn}} v]\!] )_{e}
         - \sum_{e\in \mcal E(\mcal T_h)}     ([\![ \partial_{\bm{nn}} w]\!], \{\!\{\partial_{\bm{nnn}} v\}\!\})_{e}
      \notag\\
      & - 2  \sum_{e\in \mcal E(\mcal T_h)}(\{\!\{\partial_{\bm{nnt}} w\}\!\}, [\![ \partial_{\bm{nt}} v]\!] )_{e}
        - 2 \sum_{e\in \mcal E(\mcal T_h)} ([\![ \partial_{\bm{nt}} w]\!] , \{\!\{\partial_{\bm {nnt}} v\}\!\})_{e}
      \notag\\
      & +\eta\sum_{e\in \mcal E(\mcal T_h)}   h_e^{-1}  ([\![ \partial_{\bm{nn}} w]\!] , [\![ \partial_{\bm{nn}} v]\!] )_{e}
      + \eta\sum_{e\in \mcal E(\mcal T_h)}h_e^{-3}([\![ \partial_{\bm{n}} w]\!] , [\![ \partial_{\bm{n}} v]\!] )_{e}   .
      \end{align}
      Here, $\eta>0$ is a sufficiently large constant.

      Besides, from \eqref{eq:discreteFormIntegration9}, \eqref{eq:discreteFormIntegration8} and \eqref{eq:discreteFormIntegration7}, one has
      \begin{align}\label{eq:discreteFormIntegration11}
     (\Delta^2 w, v)
      = &  \sum_{K\in \mcal T_h} (\nabla^2  w, \nabla^2 v)_K
        -    \sum_{K\in \mcal T_h} (  \partial_{\bm{nn}} w ,  \partial_{\bm{n}} v )_{\partial K}  -    \sum_{K\in \mcal T_h} (  \partial_{\bm{nt}} w ,  \partial_{\bm{t}} v )_{\partial K} .
      \end{align}
      Then, it holds
       \begin{align*}%\label{eq:discreteFormIntegration10}
        (\Delta^2 w, v)
      = &\sum_{K\in \mcal T_h} (\nabla^2  w, \nabla^2 v)_K
        - \sum_{e\in \mcal E( \mcal T_h)}   ( \{\!\{\partial_{\bm{nn}} w\}\!\}, [\![ \partial_{\bm{n}} v]\!] )_{e}
        \notag\\
        & - \sum_{e\in \mcal E( \mcal T_h)}   ( [\![ \partial_{\bm{n}} w]\!], \{\!\{\partial_{\bm{nn}} v\}\!\})_{e}
       +\eta\sum_{e\in \mcal E(\mcal T_h)} h_e^{-1}([\![ \partial_{\bm{n}} w]\!] , [\![ \partial_{\bm{n}} v]\!] )_{e}  .
      \end{align*}

For any $w,v\in V_h+H^3(\Omega)$, define
   \begin{align*}
  a_h(w,v) = &  \sum_{K \in \mcal{T}_h}(\nabla^3 w, \nabla^3 v)_K
  \notag\\
               & - \sum_{e \in \mcal{E}(\mcal T_h)}( \{\!\{\partial_{\bm{nnn}}w \}\!\}, [\![\partial_{\bm{nn}}v ]\!])_e
                - \sum_{e \in  \mcal{E}(\mcal T_h)}([\![\partial_{\bm{nn}}w ]\!], \{\!\{\partial_{\bm{nnn}}v\}\!\})_e
               \notag\\
               & - 2\sum_{e \in  \mcal{E}(\mcal T_h)}( \{\!\{\partial_{\bm{nnt}}w \}\!\},[\![\partial_{\bm{nt}}v ]\!])_e
               - 2\sum_{e \in  \mcal{E}(\mcal T_h)}([\![\partial_{\bm{nt}}w  ]\!] , \{\!\{\partial_{\bm{nnt}}v\}\!\})_e
               \notag\\
               & + \eta\sum_{e \in  \mcal{E}(\mcal T_h)} h_e^{-1}( [\![  \partial_{\bm{nn}} w  ]\!],[\![  \partial_{\bm{nn}} v ]\!])_e
                  + \eta\sum_{e \in  \mcal{E}(\mcal T_h)} h_e^{-3}( [\![  \partial_{\bm{n}} w ]\!], [\![  \partial_{\bm{n}} v ]\!])_e,\\
  b_h(w ,v) = &\sum_{K \in \mcal{T}_h}  (\nabla^2 w , \nabla^2 v)_K
                        - \sum_{e \in  \mcal{E}(\mcal T_h)} ( \{\!\{ \partial_{\bm{nn}} w \}\!\},  [\![  \partial_{\bm{n}} v ]\!])_e
                 \notag\\
                    &  - \sum_{e \in  \mcal{E}(\mcal T_h)} (  [\![  \partial_{\bm{n}} w ]\!] , \{\!\{ \partial_{\bm{nn}} v \}\!\})_e
                       + \eta\sum_{e \in \mcal{E}(\mcal T_h)} h_e^{-1}( [\![  \partial_{\bm{n}} w ]\!], [\![  \partial_{\bm{n}} v ]\!])_e.
  \end{align*}
 Here, the positive penalty parameter $\eta$ is independent of $h$ and $\iota$. Then the discretization formulation of \eqref{weakForm:GEK} is to find $w_h\in V_h$ such that
   \begin{equation}\label{eq:discreteStrangForm}
     \iota^2a_h(w_h,v_h) + b_h(w_h,v_h) =(f,v_h)\quad \forall\; v_h\in V_h. %\ell(v):=
   \end{equation}
 Define the weighted norm on $ V_h$ by
    \[
     \|  v_h\|_{\iota,h}^2  \coloneqq  \interleave v_h \interleave_{2,h}^2 +  \iota^2 \interleave v_h \interleave_{3,h}^2,
    \]
   where,
    \begin{align*}
    \interleave  v_h \interleave_{2,h}^2 \coloneqq & |  v_h |_{2,h}^2+   \sum\limits_{e\in \mcal E(\mcal T_h)}h_e^{-1}\| [\![ \partial_{\bm n }  v_h]\!]  \|_{0,e}^2 ,
    \\
    \interleave v_h \interleave_{3,h}^2 \coloneqq & |  v_h |_{3,h}^2  +  \sum\limits_{e\in \mcal E(\mcal T_h)}h_e^{-1}\|[\![ \partial_{\bm{nn} }v_h]\!] \|_{0,e}^2   +  \sum\limits_{e\in \mcal E(\mcal T_h)}h_e^{-3}\|[\![  \partial_{\bm n }  v_h ]\!] \|_{0,e}^2  ,
    \\
   |  v_h |_{2,h}^2 \coloneqq &
    \sum_{K\in \mcal T_h} |  v_h|_{2,K}^2,
   \quad
    | v_h |_{3,h}^2 \coloneqq
    \sum_{K\in \mcal T_h}|   v_h|_{3,K}^2.
    \end{align*}

 \begin{lemma}\label{lm:GEKStrangstability}
     The discrete bilinear form $\iota^2a_h(\cdot,\cdot) + b_h(\cdot,\cdot)$ is continuous on $V_h\times V_h$ and the discrete coercivity holds:
        \begin{align}\label{eq:GEKStrangstability}
     \iota^2a_h(v,v) + b_h(v,v) \geq \frac 12  \|v\|_{\iota,h}\quad \forall\; v\in V_h.
    \end{align}
  \end{lemma}
\begin{proof}
 Using the inverse inequality, one has
%        \begin{align}\label{eq:GEKstability1}
%      &\|[\![\partial_{ \bm{nt} } v]\!]\|_{0,e}\lesssim   h_e^{-1}\|[\![\partial_{\bm{n}} v ]\!]\|_{0,e} .
%    \end{align}
   %\cite{BurmanErn2007,WarburtonHesthaven2003}~
    \begin{align}
\|[\![\partial_{ \bm{nt} } v]\!]\|_{0,e}\lesssim  & h_e^{-1}\|[\![\partial_{\bm{n}} v ]\!]\|_{0,e}.
\label{eq:GEKstability1}%\\
%    \|\{\!\{\partial_{\bm{nnn}}v \}\!\}\|_{0,e}+ \|\{\!\{\partial_{\bm{nnt}}v\}\!\}\|_{0,e}\lesssim &h_e^{-1/2} \sum_{K\in \mcal T_e} | v|_{3,K}.
%    \label{eq:GEKstability2}%   \\
%h_e^{1/2}\|[\![\partial_{\bm{n}} v ]\!]\|_{0,e}\lesssim & \|\nabla v\|_{0,\omega_e},
%       \\
%h_e^{1/2}  \|[\![\partial_{\bm{nn}} v ]\!]\|_{0,e}\lesssim& \|\nabla^2 v\|_{0,\omega_e}   .
    \end{align}
   By the inverse inequality, \eqref{eq:GEKStranginverseP0} and \eqref{eq:GEKstability1}, there exists a generic constant $C_1>0$ dependent only on $\gamma_0$ such that
    \begin{align}\label{eq:GEKstability3}
       & -2 \sum_{e \in \mcal{E}(\mcal T_h)}( \{\!\{\partial_{\bm{nnn}}v\}\!\}, [\![\partial_{\bm{nn}}v ]\!])_e
        - 4\sum_{e \in  \mcal{E}(\mcal T_h)}( \{\!\{\partial_{\bm{nnt}}w \}\!\},[\![\partial_{\bm{nt}}v ]\!])_e
          \notag\\
             \leq &2\Big(\sum_{e \in \mcal{E}(\mcal T_h)}h_e\|\{\!\{\partial_{\bm{nnn}}v\}\!\}\|_{0,e}^2\Big)^{1/2} \Big( \sum_{e \in \mcal{E}(\mcal T_h)}h_e^{-1}\|[\![\partial_{\bm{nn}} v ]\!]\|_{0,e}^2\Big)^{1/2}
            \notag\\
             &+ 4\Big(\sum_{e \in \mcal{E}(\mcal T_h)}h_e\|\{\!\{\partial_{\bm{nnt}}v\}\!\}\|_{0,e}^2\Big)^{1/2}\Big(\sum_{e \in \mcal{E}(\mcal T_h)}h_e^{-1}\|[\![\partial_{\bm{nt}} v ]\!]\|_{0,e}^2\Big)^{1/2}
              \notag\\
              \leq &\frac 12 |v|_{3,h}^2 + C_1 \sum_{e \in \mcal{E}(\mcal T_h)}\Big( h_e^{-1}\|[\![\partial_{\bm{nn}} v ]\!]\|_{0,e}^2 +h_e^{-3}\|[\![\partial_{\bm{n}} v ]\!]\|_{0,e}^2\Big).
    \end{align}
    Similarly, it follows from \eqref{eq:GEKStranginverseP0} and the Cauchy-Schwarz inequality that
  \begin{align} \label{eq:GEKstability5}
        - 2\sum_{e \in  \mcal{E}(\mcal T_h)} ( \{\!\{ \partial_{\bm{nn}}v \}\!\},  [\![  \partial_{\bm{n}} v ]\!])_e
                   %\lesssim  & \sum_{e \in  \mcal{E}(\mcal T_h)} h_e^{-1/2}\| [\![\partial_{\bm{n}} v ]\!]\|_{0,e}\sum_{K\in \mcal T_e}\| \nabla^2  v \|_{0,K}
       % \notag\\
        \leq & \frac 12 |v|_{2,h}^2  + C_2\sum_{e \in \mcal{E}(\mcal T_h)} h_e^{-1}\|[\![\partial_{\bm{n}} v ]\!]\|_{0,e}^2  ,
  \end{align}
 where $C_2>0$ is a generic constant dependent only on $\gamma_0$.  Combining with \eqref{eq:GEKstability3}-\eqref{eq:GEKstability5}, we have
      \begin{align*}
  \iota^2 a_h(v,v)  + b_h(v ,v)
  \geq & \frac 12  \iota^2|v|_{3,h}^2 +\iota^2(\eta- C_1 )\sum_{e \in \mcal{E}(\mcal T_h)} h_e^{-1}\|[\![\partial_{\bm{nn}} v ]\!]\|_{0,e}^2
   \notag\\
   &+\iota^2 (\eta- C_1 ) \sum_{e \in \mcal{E}(\mcal T_h)} h_e^{-3}\|[\![\partial_{\bm{n}} v ]\!]\|_{0,e}^2
   \notag\\
   &+ \frac 12  |v|_{2,h}^2 + (\eta- C_2 ) \sum_{e \in \mcal{E}(\mcal T_h)} h_e^{-1}\|[\![\partial_{\bm{n}} v ]\!]\|_{0,e}^2.
  \end{align*}
  The discrete coercivity is derived by setting $\eta$ as $\frac 12 + \max\{C_1,C_2\}$ in the last inequality.

  We omit the proof of the continuity of $\iota^2a_h(\cdot,\cdot) + b_h(\cdot,\cdot)$, which is easy to check by \eqref{eq:GEKstability1}, the inverse inequality and the Cauchy-Schwarz inequality.
\end{proof}

\begin{remark}
  As shown in \cite{Hjelle2006,HuangLai2013}, we count the global DOFs in view of the relation:
  \[
  \lim\limits_{h\to 0} \frac{N_e}{N_v} = 3,\quad \text{and} \quad \lim\limits_{h\to 0} \frac{N_T}{N_v} = 2.
  \]
where $N_e = \#\mathcal E(\mathcal T_h)$, $N_v=\#\mathcal V(\mathcal T_h)$, $N_T = \#\mathcal T_h$. In \cite{Gudi2011sixthIPDG}, Gudi use the cubic Lagrange element to approximate the sixth-order problem, which has the same number of the local DoFs as the Hermite element in a triangle. However, the number of the global DoFs for the cubic Lagrange element space is $9N_v$, while one for $V_h$ is $5N_v$.
\end{remark}

\section{Error analysis}\label{Sect:InterpolationError}

\subsection{The quasi interpolation operator}
%By a simple manipulation, the nodal basis functions of the Hermite element are defined by
%   \begin{align*}
%     \left\{
%        \begin{aligned}
%         &r_0 = 27b_K,
%          \\
%          &r_i = 3\lambda_i^2 - 2\lambda_i^3 - 7b_K,%\quad \quad 1\leq i \leq 3,
%          \\
%          &r_{ij} = \lambda_i\lambda_j(2\lambda_i + \lambda_j - 1),%\quad 1\leq i\not= j \leq 3.
%          \\
%        \end{aligned}
%        \right.
%   \end{align*}
%   with $1\leq i\not= j \leq 3$. % Then, it holds for any $p\in \mbb P_3(K)$ that
%% \[
%%    p = p(a_0) r_0 + \sum_{i=1}^3p(a_i) r_i + \sum_{i,j=1\atop i\not=j}^3  \nabla p( a_i)(a_j-a_i)  r_{ij}.
%%   \]
% \begin{equation}\label{lm:GEK16Interpolation}
%  I_h v = v(a_0) r_0 + \sum_{i=1}^3 v(a_i) r_i + \sum_{i,j=1\atop i\not=j}^3 \sum_{K\in \omega_{a_i}} \nabla \Pi_K^3v( a_i)(a_j-a_i)  r_{ij}.
%\end{equation}
In order to derive some robust and optimal error estimates, we first define an weak interpolation operator, which is motivated by the definition of the average projection operators in \cite[Sect. 3.5]{Shi-Wang-2013}. Denote the weak interpolation operator from $H^2_0(\Omega )$ to $V_h$ by $I_h$. Then, for any $v\in H^2_0(\Omega)$, the DoFs of $I_h v$ are given as:
   \begin{align}
       I_hv(\zeta)  & =\frac {1}{\#\mcal T_\zeta}\sum_{K'\in \mcal T_{\zeta}}  (\Pi_{K'}^3v)(\zeta ),
      \label{weakInterpolation1}  \\
        \nabla I_h v(\zeta)  &= \frac {1}{\#\mcal T_{\zeta}} \sum_{K'\in \mcal T_{\zeta}}   (\nabla \Pi_{K'}^3v)(\zeta),
        \label{weakInterpolation2}\\
     I_h v(\zeta_K ) & =v(\zeta_K ),%\quad a_K\text{ is the barycenter of $K$},
  \label{weakInterpolation3}
      \end{align}
      where $\zeta\in \mathcal V^i(\mathcal T_h)$, $K\in \mathcal T_h$ and DoFs \eqref{weakInterpolation1}-\eqref{weakInterpolation2} vanish on boundary. Following the ideas for proving Theorem 3.5.3 in \cite{Shi-Wang-2013}, we have
 \begin{equation}\label{eq:GEKP3InterpolationError}
    |v-I_h v|_{m,h}\lesssim h ^{s-m}|v|_{s  }\quad\forall \; v\in H^{s}(\Omega)\cap H^2_0(\Omega),\, 0\leq m \leq s ,\, 2\leq s\leq 4.
  \end{equation}
%{\color{red} Noted that if we use the cubic Lagrange finite element space to discrete form of problem \eqref{eq:discreteStrangForm}, the estimate \eqref{eq:GEKP3InterpolationError} still valid for the nodal interpolation operator $I_h$ related to the aforementioned space.}

%
%With the help of the regularity result \eqref{eq:GEKRegularity} and \eqref{eq:GEKP3InterpolationError}, we can derive the robust interpolation error estimates in the weighted norm $\|\cdot\|_{\iota,h}$.
 \begin{lemma}\label{lm:GEKStrangInterpolation}
  Let $w\in V\cap H^4(\Omega)$ and $w_0\in H^2_0(\Omega)\cap H^s(\Omega)$ $(s=3,4)$ be the solutions of problems \eqref{weakForm:GEK} and \eqref{weakForm:KPB}, respectively. Then, it holds
   \begin{align}
       \|w-I_h w\|_{\iota,h}&\lesssim \iota^{1/2}\|f\|_0 + h^{s-2} |w_0|_s,\label{eq:GEKInterEf}\\
	   \|w  -I_h w\|_{\iota,h} &\lesssim (\iota h^{s-3} + h^{s-2}  )   |w|_s .\label{eq:GEKInterEw}
   \end{align}
 \end{lemma}
\begin{proof}
 By the trace inequality \eqref{eq:trace1}, we have
 \begin{align*}
&  \sum_{e\in \mcal E(\mcal T_h)}h_e^{-1}\| [\![\partial_{\bm{nn}}(w-I_h w) ]\!]\|_e^2 + \sum_{e\in \mcal E(\mcal T_h)}h_e^{-3}\| [\![\partial_{\bm{n}}(w-I_h w) ]\!]\|_e^2
\notag\\
  \lesssim &\sum_{K\in \mcal   T_h }\big(h_K^{-4} |w-I_h w|_{1,K}^2 + h_K^{-2} |w-I_h w|_{2,K}^2  +  |w-I_h w|_{3,K} ^2    \big),
 \end{align*}
which combined with \eqref{eq:GEKP3InterpolationError} and \eqref{eq:GEKRegularity} yields
 \begin{align}
    \iota^2   \interleave w-I_h w \interleave_{3,h}^2
   \lesssim &   \iota^2 |w|_3^2\lesssim \iota \|f\|_0 ^2,
   \label{eq:GEKInterEf1}\\
    \iota^2   \interleave w-I_h w \interleave_{3,h}^2
   \lesssim &     \iota^2h^{2s-6}|w|_s^2.\label{eq:GEKInterEw1}
 \end{align}
 Similarly, employing the trace inequality \eqref{eq:trace1} and \eqref{eq:GEKP3InterpolationError}, one gets
  \begin{align} \label{eq:GEKInterEw2}
      \interleave w-I_h w \interleave_{2,h}^2
      \lesssim &\sum_{K\in \mcal   T_h } \big(h_K^{-2} |w-I_h w|_{1,K} ^2+   |w-I_h w|_{2,K} ^2  \big)
       \lesssim   h^{2s-4} |  w|_{4} ^2.
 \end{align}
 On the other hand, combining the trace inequality \eqref{eq:trace1}, \eqref{eq:GEKP3InterpolationError} and \eqref{eq:GEKRegularity} leads to
 \begin{align} \label{eq:GEKInterEf2}
      \interleave w-I_h w \interleave_{2,h}^2
      \lesssim &\sum_{K\in \mcal   T_h } \big(h_K^{-2} |w-I_h w|_{1,K} ^2+   |w-I_h w|_{2,K} ^2  \big)
      \notag\\
       \lesssim &\sum_{K\in \mcal   T_h } \big(h_K^{-2} |w-w_0 -I_h (w-w_0)|_{1,K}^2 +   |w-w_0 -I_h (w-w_0)|_{2,K}^2   \big)
       \notag\\
      & +\sum_{K\in \mcal   T_h } \big(h_K^{-2} |w_0 -I_hw_0 |_{1,K}^2 +   |w_0 -I_h w_0|_{2,K}^2   \big)
     \notag\\
       \lesssim &   |w-w_0|_2^2+ h^{2s-4} |w_0|_s^2    \lesssim   \iota \|f\|_0 ^2+ h^{2s-4} |w_0|_s^2.
 \end{align}
Finally, the estimate \eqref{eq:GEKInterEf} can be derived by \eqref{eq:GEKInterEf1} and \eqref{eq:GEKInterEf2}, while the estimate \eqref{eq:GEKInterEw} is derived by \eqref{eq:GEKInterEw1} and \eqref{eq:GEKInterEw2}.
\end{proof}

\subsection{The enriching operator}

To obtain the error estimate under minimal regularity of $w$, we also need to construct an enriching operator connecting $V_h$ to a $H^3$ conforming finite element space. Hence, we first introduce a $C^2$-finite element (cf. \cite{BrambleZlamal1970,Zenisek1970C2element}). Its shape function space is chosen as $V^c(K)\coloneqq\mbb P_9(K)$ and the DoFs consist of:
%  \begin{equation}\label{C2P9dof:enrichingOperator}
%          \left\{
%        \begin{aligned}
%        & \nabla^{\alpha} v(a_i), \;\alpha=1,2,3,4,\\
%        &  \partial_{\bm{nn}} v(a_{e_l,1}) ,\partial_{\bm{nn}} v(a_{e_l,2}),  \\
%         &  \partial_{\bm{n}} v(a_{e_l}),  \\
%         & v( a_0),
%        \end{aligned}
%        \right.
%      \end{equation}
%Here, $a_{e}= ( a_i+a_j)/2$ is the midpoint of $e_l$, $a_{e_l,1}=(a_i+ 2 a_j)/3$ and $a_{e_l,2}=(a_i+ 2 a_j)/3$ with, $1\leq i<j\leq 3$, $i,j\not= l$, $1\leq l\leq3$.
  \begin{equation}\label{C2P9dof:enrichingOperator}
          \left\{
        \begin{aligned}
        & \nabla^{\alpha} v(\zeta)\quad  \forall \; \zeta\in \mcal V(K),\;\alpha=0,1,2,3,4,\\
        &  \partial_{\bm{nn}} v(\zeta_{e,1}) ,\partial_{\bm{nn}} v(\zeta_{e,2})\quad  \forall \; e\in \mcal E(K),  \\
         &  \partial_{\bm{n}} v(\zeta_{e})\quad  \forall \; e\in \mcal E(K)\\
         & v(\zeta_K).
        \end{aligned}
        \right.
      \end{equation}
  Here, $\zeta_K$ is the barycenter of $K$ and for each edge $e=\overline{a_1^ea_2^e}\in \mcal E(K)$ with $ a_1^e,a_2^e\in \mcal V(K)$, $\zeta_{e}= (a_1^e+a_2^e)/2$ is the midpoint of $e$, $\zeta_{e,1}=(a_1^e+ 2a_2^e  )/3$, $\zeta_{e,2}=(2 a_1^e+a_2^e)/3$. Set
   \begin{align*}
      W_h^c = &\{v\in H^3(\Omega):\, v|_K \in V^c(K) \,\text{for all $K\in \mathcal T_h$, $\nabla^{\alpha}v$ is continuous at all vertices}\\
     & \quad\quad \quad  \quad\text{of triangulation $\mathcal T_h$ with $\alpha =0,1,2,3,4$, $\partial_{\bm n}v $ is continuous at the}\\
     & \quad \quad \quad  \quad\text{midpoints of $e\in \mathcal E(K)$, $\partial_{\bm{nn}} v$ is continuous at the points of}\\
     & \quad\quad \quad \quad \text {trisection of $e\in \mathcal E(K)$and all DoFs \eqref{C2P9dof:enrichingOperator} vanish on boundary}\}.
     \end{align*}
    Then, it holds $ W_h^c\subset H^3_0(\Omega)$ and
       \begin{equation}\label{eq:GEKJump}
   [\![\partial_{\bm{nn}}  v  ]\!]|_e =[\![\partial_{\bm{n}}  v  ]\!] |_e= [\![\partial_{\bm{nt}}  v  ]\!]|_e= [\![\partial_{\bm{tt}}  v  ]\!]|_e =0\quad \forall\;v\in W_h^c,\,e\in \mcal E(\mcal T_h).
  \end{equation}
  The enriching operator $E^c_h : V_h \to W_h^c$ is defined as follows:
            \begin{align}
        \nabla^{\alpha}E_h^c v(\zeta)  &= \frac {1}{\#\mcal T_{\zeta}} \sum_{K'\in \mcal T_{\zeta}}   (\nabla^{\alpha}v|_{K'})(\zeta), \quad\alpha=0,1,2,3,4,
        \label{C2P9dof:enrichingOperator1}\\
     \partial_{\bm{nn}}E_h^c v(\zeta_{e,i})& =\frac {1}{\#\mcal T_e}\sum_{K'\in \mcal T_e}\partial_{\bm{nn}}(v|_{K'})(\zeta_{e,i}), \quad i=1,2,
      \label{C2P9dof:enrichingOperator2}  \\
     \partial_{\bm{n}}E_h^c v(\zeta_e)& =\frac {1}{\#\mcal T_e}\sum_{K'\in \mcal T_e}\partial_{\bm{n}}(v|_{K'})(\zeta_e),
     \label{C2P9dof:enrichingOperator3}\\
     E_h^c v(\zeta_K ) & =v(\zeta_K ),%\quad a_K\text{ is the barycenter of $K$},
     \label{C2P9dof:enrichingOperator4}
      \end{align}
where, $K\in \mcal T_h$, $e\in \mcal E^i(\mcal T_h)$ and $\zeta\in \mcal V^i(\mcal T_h)$. Noted that the DoFs of $E_h^c v$ vanish on boundary $\partial\Omega$.

%Applying the similar argument for proving Lemma 2.1 in \cite{Gudi2011sixthIPDG}, the following lemma holds.
\begin{lemma}
For any $v\in V_h$, it holds
  \begin{align}
  \sum_{K\in \mcal T_h}h_K^{-4}  \|v - E_h^c v\|_{0,K}^2+  \sum_{K\in \mcal T_h} h_K \|\partial_{\bm{nn}}(v - E_h^c v)\|_{0,\partial K}^2\quad&
  \notag\\
+   \sum_{K\in \mcal T_h}h_K^{-1} \|\partial_{\bm{n}}(v - E_h^c v)\|_{0,\partial K}^2  \lesssim& \interleave  v_h \interleave_{2,h}^2,
   \label{eq:GEKenrichingError1}  \\
   \sum_{K\in \mcal T_h}  h_K^{-6}  \|v - E_h^c v\|_{0,K}^2+  \sum_{K\in \mcal T_h}h_K  \|\partial_{\bm{nnn}}(v - E_h^c v)\|_{0,\partial K}^2\quad\quad\quad\quad\quad &
\notag\\
+  \sum_{K\in \mcal T_h}h_K^{-1} \|\partial_{\bm{nn}}(v - E_h^c v)\|_{0,\partial K}^2+   \sum_{K\in \mcal T_h}h_K^{-3} \|\partial_{\bm{n}}(v - E_h^c v)\|_{0,\partial K}^2\lesssim &\interleave  v_h \interleave_{3,h}^2,
  \label{eq:GEKenrichingError2} \\
\|E_h^c v \|_{\iota,h} \lesssim &\| v \|_{\iota,h} .
\label{eq:GEKenrichingError3}
\end{align}
\end{lemma}
\begin{proof}
  By \eqref{C2P9dof:enrichingOperator1}-\eqref{C2P9dof:enrichingOperator4} and scaling argument, we have
  \begin{align}\label{GEK:enrichError1}
    \|  v  - E_h^c v\|_{0,K}^2 \lesssim
    & h_K^{2+2\alpha} \sum_{\alpha=2}^4 \sum_{\zeta\in \mcal V(K)}\big|\nabla^{\alpha} v|_K - \nabla^{\alpha}E_h^c v \big|^2(\zeta )
    +  h_K^{4} \sum_{e\in \mcal E(K)}\big|\partial_{\bm{n}} v|_K - \partial_{\bm{n}}E_h^c v \big|^2(\zeta_{e} )
    \notag\\
    &
    +h_K^{6} \sum_{i=1}^2\sum_{e\in \mcal E(K)}\big|\partial_{\bm{nn}} v|_K - \partial_{\bm{nn}}E_h^c v \big|^2(\zeta_{e,i} )
    \notag\\
    \lesssim
    & h_K^{1+2\alpha} \sum_{\alpha=2}^4 \sum_{\zeta\in \mcal V(K)} \sum_{e\in \mcal E(\mcal T_h),\zeta\in \partial e}\|[\![  \nabla^{\alpha} v ]\!]\|_{0,e}^2
     +  h_K^{3} \sum_{e\in \mcal E(K)}\|[\![ \partial_{\bm{n}} v]\!]\|_{0,e}^2
      \notag\\
      &+  h_K^{5} \sum_{e\in \mcal E(K)}\|[\![ \partial_{\bm{nn}} v]\!]\|_{0,e}^2,
  \end{align}
  Applying the inverse inequality, \eqref{GEK:enrichError1} and \eqref{eq:GEKStranginverseP0}, we can get
   \begin{align}\label{GEK:enrichError2}
% &  h_K^{2m} \|  v  - E_h^c v\|_{m,K}^2 \lesssim
   \|  v  - E_h^c v\|_{0,K}^2
    \lesssim
    & h_K^{2\alpha} \sum_{\alpha=2}^4 \sum_{\zeta\in \mcal V(K)}\sum_{K'\in \mcal T_{\zeta}}   \| \nabla^{\alpha} v  \|_{0,K'}^2
     +  h_K^{3} \sum_{e\in \mcal E(K)}\|[\![ \partial_{\bm{n}} v]\!]\|_{0,e}^2
      \notag\\
       \lesssim
    & h_K^{4}  \sum_{\zeta\in \mcal V(K)} \sum_{K'\in \mcal T_{\zeta}}    | v  |_{2,K'}^2
     +  h_K^{3} \sum_{e\in \mcal E(K)}\|[\![ \partial_{\bm{n}} v]\!]\|_{0,e}^2.
  \end{align}

  On the other hand, due to the fact that $V_h\in H^1_0(\Omega)\cap C(\overline{\Omega})$ and using \eqref{eq:discreteFormIntegration2} and the inverse inequality, we can obtain
  \begin{align*}
    &\|[\![ \nabla^2v ]\!]\|_{0,e} \lesssim  \|[\![ \partial_{\bm{nn}} v]\!]\|_{0,e} + \|[\![ \partial_{\bm{nt}} v]\!]\|_{0,e} \lesssim  \|[\![ \partial_{\bm{nn}} v]\!]\|_{0,e} + h_e^{-1}\|[\![ \partial_{\bm{n} }v]\!]\|_{0,e} ,
  \end{align*}
  and
     \begin{align*}
     \|  \nabla^4 v \|_{0,K} \lesssim  h_K^{-1}\|  \nabla^3v \|_{0,K}   ,
  \end{align*}
 which together with \eqref{GEK:enrichError1} yields
  \begin{align}\label{GEK:enrichError21}
  %& h_K^{2m} \|  v  - E_h^c v\|_{m,K}^2 \lesssim
    \|  v  - E_h^c v\|_{0,K}^2
  %  \lesssim
%    & h_K^{2\alpha} \sum_{\alpha=3}^4 \sum_{\zeta\in \mcal V(K)}  \sum_{K'\in \mcal T_{\zeta}}  \| \nabla^{\alpha} v  \|_{0,K'}^2
%     +  h_K^{3} \sum_{e\in \mcal E(K)}\|[\![ \partial_{\bm{n}} v]\!]\|_{0,e}^2
%      \notag\\
%      &+  h_K^{5} \sum_{e\in \mcal E(K)}\|[\![ \partial_{\bm{nn}} v]\!]\|_{0,e}^2
%      \notag\\
       \lesssim
    & h_K^{6}  \sum_{\zeta\in \mcal V(K)}   \sum_{K'\in \mcal T_{\zeta}}| v  |_{3,K'}^2
     +  h_K^{3} \sum_{e\in \mcal E(K)}\|[\![ \partial_{\bm{n}} v]\!]\|_{0,e}^2
      \notag\\
      &+  h_K^{5} \sum_{e\in \mcal E(K)}\|[\![ \partial_{\bm{nn}} v]\!]\|_{0,e}^2.
  \end{align}

Summing over all triangles in $\mcal T_h$ and employing \eqref{eq:GEKStranginverseP0} and the inverse inequality, we can derive the estimate \eqref{eq:GEKenrichingError1} from \eqref{GEK:enrichError2}, and we can derive the estimate \eqref{eq:GEKenrichingError2} from \eqref{GEK:enrichError21}. Moreover, the inequality \eqref{eq:GEKenrichingError3} can be obtained by \eqref{eq:GEKenrichingError1}, \eqref{eq:GEKenrichingError2} and the inverse inequality. %It suffices to prove the first two inequalities.

\end{proof}

\subsection{Error estimates}\label{Sect:ErrorAnalysis}
%\subsection{A priori error estimates}

\begin{lemma} \label{lm:GEKerrorStrang}
For any $e\in \mcal E^i(\mcal T_h)$ and $v_h\in V_h$, it holds
  \begin{align}
    %\|f-\Delta^2 v_h\|_{0,K}& \lesssim\|f-\Pi_K^0f\|_{0,K} +  \iota^2h_K^{-3}|w-v_h|_{3,K} + h_K^{-2}|w-v_h|_{2,K}
    \|f\|_{0,K} \lesssim&\|f-\Pi_K^0f\|_{0,K} +  \iota^2h_K^{-3}|w-v_h|_{3,K} + h_K^{-2}|w-v_h|_{2,K},
    \label{eq:GEKerrorf_Deltavh}
    \\
    \iota^{4}\| [\![  \partial_{\bm{nnn}} v_h ]\!]\|_{0,e}^2 \lesssim & h_e^{5} \sum_{K\in \mcal T_e} \|f-\Pi_K^0f\|_{0,K}^2  +  \iota^4h_e^{-1} \sum_{K\in \mcal T_e}|w-v_h|_{3,K}^2
    \notag\\
    &+h_e  \sum_{K\in \mcal T_e}|w-v_h|_{2,K}^2,
    \label{eq:GEKerrorVhnnn}
      \\
    \| [\![  \partial_{\bm{nn}} v_h ]\!]\|_{0,e}^2 \lesssim& h_e^{3} \sum_{K\in \mcal T_e}\|f-\Pi_K^0f\|_{0,K}^2   +  \iota^4 h_e^{-3}\sum_{K\in \mcal T_e}|w-v_h|_{3,K}^2
    \notag\\
    & + h_e^{-1} \sum_{K\in \mcal T_e}|w-v_h|_{2,K}^2.
    \label{eq:GEKerrorVhnn}
  \end{align}
\end{lemma}
\begin{proof}
  Set
    \begin{align*}
    \phi_K&=
        \begin{cases}
 b_K^3 \Pi_K^0f & \textrm{in}~K ,\\
          0&\textrm{in}~\Omega\backslash K.
        \end{cases}
  \end{align*}
  Obviously, $\phi_K|_K\in H^3_0(K)$ and we have from scaling argument that
  \begin{align}\label{eq:GEKphiK}
    \|\phi_K\|_{0,K}^2\lesssim \|\Pi_K^0f\|_0^2\lesssim(\Pi_K^0f , \phi_K)_K .
  \end{align}
  By integration by parts and \eqref{weakForm:GEK}, it holds for any $v_h \in V_h$ that
  \begin{align*}%\label{eq:GEKPifDeltavh}
    \|\Pi_K^0f\|_{0,K}^2 \lesssim &(\Pi_K^0f-f, \phi_K)_K+ ( f-\Delta^2 v_h +\iota^2\Delta^3v_h, \phi_K)_K
  \notag\\
%    =& (\Pi_K^0f-f, \phi_K)_K+\iota^2 ( \nabla^3w,\nabla^3\phi_K)_K -\iota^2(\Delta^3v_h, \phi_K)_K
%     \notag\\
%     &+ (\nabla^2w,\nabla^2\phi_K)_K- (\Delta^2 v_h, \phi_K)_K
%     \notag\\
     =&(\Pi_K^0f-f, \phi_K)_K + \iota^2(\nabla^3w-\nabla^3v_h,\nabla^3\phi_K)_K  + (\nabla^2w-\nabla^2v_h,\nabla^2\phi_K)_K
     %\notag\\
%     \lesssim&  \|\Pi_K^0 f-f\|_{0,K}\|\phi_K\|_{0,K} + \iota^2 |w-v_h|_{3,K} |\phi_K|_{3,K}+ |w-v_h|_{2,K} |\phi_K|_{2,K}
.
  \end{align*}
  Then, using the inverse inequality and the Cauchy-Schwarz inequality, we can obtain
    \[
    \|\Pi_K^0f\|_{0,K}^2\lesssim ( \|f-\Pi_K^0f\|_{0,K}+\iota^2h_K^{-3}|w-v_h|_{3,K} + h_K^{-2}|w-v_h|_{2,K})\|\phi_K\|_{0,K}.
  \]
We get \eqref{eq:GEKerrorf_Deltavh} from \eqref{eq:GEKphiK} and the triangle inequality readily.

  For each interior edge $e\in\mcal E^i(\mcal T_h)$, there exist two adjacent triangles, $K_1$ and $K_2$ in $\mcal T_h$ sharing a common edge $e$ and define a bubble function as follow:
  \begin{align*}
    \varrho_{e} &=
        \begin{cases}
  \lambda_{K_1,1}\lambda_{K_1,2} \lambda_{K_2,1}\lambda_{K_2,2} & \textrm{in}~\omega_e ,\\
          0&\textrm{in}~\Omega\backslash \omega_e ,
        \end{cases}
    %    \notag\\
%       \varrho_{e,2} &=
%        \begin{cases}
%%      b_{K_1} & \textrm{ÔÚ}~K_1~\textrm{ÖÐ},\\
%%     b_{K_2} & \textrm{ÔÚ}~K_2~\textrm{ÖÐ},\\
%       b_{K_1}  b_{K_2} & \textrm{ÔÚ}~\omega_e~\textrm{ÖÐ},\\
%          0&\textrm{ÔÚ}~\Omega\backslash \omega_e~\textrm{ÖÐ},
%            \end{cases}
  \end{align*}
 where, $\lambda_{K_1,i}$ and $\lambda_{K_2,i}$ are barycentric coordinates of $K_1$ and $K_2$ related to the vertices of $e$, $i=1,2$. Besides, the line where $e$ lying in can be represented by $\bm l_e =\bm n_e\cdot \bm x + C_e $ with a constant $C_e$. A simple calculation shows $\bm l_e |_e =0$, $\partial_{\bm n_e}\bm l_e |_e =1$ and
 \begin{equation}\label{eq:GEKeq_e}
 \big|\bm l_{e} |_{\omega_e}\big| = h_e\Big|\frac{( \bm n_e\cdot \bm x + C_e)|_{\omega_e}}{h_e}\Big| \lesssim h_e.
 \end{equation}

 We now prove \eqref{eq:GEKerrorVhnnn}. Define $J_{1,\omega_e}$ by extending the jump $[\![  \partial_{\bm{nnn}} v_h ]\!]|_e$ to $\omega_e$ constantly along the normal to $e$. Take $\phi_e =\varrho_{e}^3J_{1,\omega_e} \bm l_e ^2  $. Then, $\phi_e  \in H^3_0(\omega_e)$, $\phi_e |_e=\partial_{\bm n_e}\phi_e|_e= 0$ and
\begin{align*}% \label{eq:GEKnnnB1}
  &\partial_{\bm{nn}} \phi_e|_e
  = (\varrho_{e}^3J_{1,\omega_e} \partial_{\bm{nn}}  \bm l_e ^2 )|_e = 2\varrho_{e}^3[\![  \partial_{\bm{nnn}} v_h ]\!]|_e.
% \\
%&\partial_{\bm{tt}} \phi_e|_{\partial K_1\cup \partial K_2 } =\partial_{\bm{nt}} \phi_e|_{\partial K_1\cup \partial K_2 } = \partial_{\bm n} \phi_e|_{\partial K_1\cup \partial K_2 }=\phi_e|_{\partial K_1\cup \partial K_2 }=0.
%  \label{eq:GEKnnnB2}
\end{align*}
 By integrating by parts, \eqref{model:GEK}, scaling argument and the similar arguments for obtaining \eqref{eq:discreteFormIntegration5}, it holds for any $e\in \mcal E^i(\mcal T_h)$ that
   % \begin{align*}
%       \iota^2 \| [\![  \partial_{\bm{nnn}_e} v_h ]\!]\|_{0,e}^2\lesssim&   \iota^2  ([\![  \partial_{\bm{nnn}_e} v_h ]\!], \partial_{\bm{nn}}\phi_e)_e  =   \iota^2  \sum_{K\in \mcal T_e}(\partial_{\bm{nnn}}v_h , \partial_{\bm{nn}}\phi_e)_{\partial K}
%    \notag\\&
%    =     \iota^2 \sum_{K\in \mcal T_e}  (\nabla^3(v_h-w),\nabla^3\phi_e)_{K} + (f-\Delta^2v_h, \phi_e)_{\omega_e} -  (\nabla^2 w ,\nabla^2\phi_e)_{\omega_e},
%  \end{align*}
   \begin{align}\label{eq:GEKnnnB2}
       \iota^2 \| [\![  \partial_{\bm{nnn}} v_h ]\!]\|_{0,e}^2\lesssim&   \iota^2  ([\![  \partial_{\bm{nnn}} v_h ]\!], \partial_{\bm{nn}}\phi_e)_e  =   \iota^2  \sum_{K\in \mcal T_e}(\partial_{\bm{nnn}}v_h , \partial_{\bm{nn}}\phi_e)_{\partial K}
    \notag\\
    = &    \iota^2 \sum_{K\in \mcal T_e}  (\nabla^3(v_h-w),\nabla^3\phi_e)_{K} + (f, \phi_e)_{\omega_e} -(\Delta^2 w , \phi_e)_{\omega_e}
    \notag\\
    =& \iota^2 \sum_{K\in \mcal T_e}  (\nabla^3(v_h-w),\nabla^3\phi_e)_{K} + (f, \phi_e)_{\omega_e}
    \notag\\
    &-  \sum_{K\in \mcal T_e} (\nabla^2 v_h ,\nabla^2\phi_e)_{K} + \sum_{K\in \mcal T_e} (\nabla^2 (v_h-w) ,\nabla^2\phi_e)_{K},
  \end{align}
 Noted that  $v_h|_K\in \mathbb P_3(K)$, $\phi_e\in H^3_0(\Omega)$ and $\phi_e |_e=\partial_{\bm n_e}\phi_e|_e= 0$. Using integration by parts, we have
  \[
  \sum_{K\in \mcal T_e} (\nabla^2 v_h ,\nabla^2\phi_e)_{K} = \sum_{K\in \mcal T_e}(\Delta^2 v_h,\phi_e)_{K}  + \sum_{\hat{e}\in \mcal \partial K,K\in \mcal T_e}\big( (\nabla^2 v_h ,\partial_n\phi_e)_{\hat{e}} -  (\partial_n\Delta v_h , \phi_e)_{\hat{e}}\big)=0,
  \]
which combined with \eqref{eq:GEKnnnB2} and the inverse inequality implies
      \begin{align}\label{eq:GEKVhnnn}
  & \iota^2 \| [\![  \partial_{\bm{nnn}} v_h ]\!]\|_{0,e}^2 \notag\\%\lesssim&\iota^2 (\nabla^3(v_h-w),\nabla^3\phi_e)_{\omega_e} +  (f-\Delta^2v_h, \phi_e)_{\omega_e} -  (\nabla^2 w ,\nabla^2\phi_e)_{\omega_e}
  % \notag\\
   %&=
 %  \iota^2 (\nabla^3(v_h-w),\nabla^3\phi_e)_{\omega_e}+  (\nabla^2( v_h-w ) ,\nabla^2\phi_e)_{\omega_e} +  (f, \phi_e)_{\omega_e}
%     \notag\\
   \lesssim &  \iota^2  \sum_{K\in \mcal T_e}|v_h-w|_{3,K}|\phi_e|_{3,K}   + \sum_{K\in \mcal T_e}|v_h-w|_{2,K}|\phi_e|_{2,K} +\|f \|_{0,\omega_e}  \|\phi_e\|_{0,\omega_e}
   \notag\\
   \lesssim &( \iota^4h_e^{-6}\sum_{K\in \mcal T_e} |v_h-w|_{3,K}^2   +h_e^{-4}\sum_{K\in \mcal T_e}  |v_h-w|_{2,K}^2 +\|f \|_{0,\omega_e}^2  )^{1/2}\|\phi_e\|_{0,\omega_e}.
  \end{align}
  From \eqref{eq:GEKeq_e} and the inverse inequality, one has
    \begin{align}\label{eq:GEKphie}
    \| \phi_e\|_{0,\omega_e}\lesssim h_e ^2\| J_{1,\omega_e}\|_{0,\omega_e} \lesssim h_e^{3}\|[\![  \partial_{\bm{nnn}} v_h ]\!]\|_{\infty,e}\lesssim h_e^{5/2}\|[\![  \partial_{\bm{nnn}} v_h ]\!]\|_{0,e}.
  \end{align}
  Therefore, \eqref{eq:GEKerrorVhnnn} can be achieved by \eqref{eq:GEKerrorf_Deltavh}, \eqref{eq:GEKVhnnn} and \eqref{eq:GEKphie}.

 Finally, we are going to prove \eqref{eq:GEKerrorVhnn}. For any $e\in \mcal E^i(\mcal T_h)$,  define $J_{2,\omega_e}$ by extending the jump $[\![  \partial_{\bm{nn}} v_h ]\!]|_e$ to $\omega_e$ constantly along the normal to $e$. Set $\varphi_e =  \varrho_{e,2}^3J_{2,\omega_e}\bm l_e\in H^3_0(\omega_e)$ and it's easy to know that
  \[
    \partial_{\bm n_e}\varphi_e|_e =  \varrho_{e,2}^3 [\![  \partial_{\bm{nn}} v_h ]\!]|_e.%\quad \partial_{\bm n}\varphi_e|_{\partial K_1\cup \partial K_2\backslash e} =0.
  \]
  We get from the inverse inequality and \eqref{eq:GEKeq_e} that
  \begin{align}\label{eq:GEKvarphie}
    \| \varphi_e\|_{0,\omega_e}\lesssim h_e \| J_{2,\omega_e}\|_{0,\omega_e} \lesssim h_e^{2}\|[\![  \partial_{\bm{nn}} v_h ]\!]\|_{\infty,e}\lesssim  h_e^{3/2}\|[\![  \partial_{\bm{nn}} v_h ]\!]\|_{0,e} .
  \end{align}
  By scaling argument and the similar arguments for obtaining \eqref{eq:discreteFormIntegration11},
 \begin{align*}
    \| [\![  \partial_{\bm{nn}} v_h ]\!]\|_{0,e}^2\lesssim& ([\![  \partial_{\bm{nn}} v_h ]\!], \partial_{\bm n_e}\varphi_e)_e
    = \sum_{K\in \mcal T_e}(\partial_{\bm{nn}}v_h , \partial_{\bm n}\varphi_e)_{\partial K}
    =   \sum_{K\in \mcal T_e} (\nabla^2v_h,\nabla^2\varphi_e)_{K}.
  \end{align*}
  Then, using integration by parts and \eqref{model:GEK}, it holds
   \begin{align}\label{eq:GEKnnvh}
  \| [\![  \partial_{\bm{nn}} v_h ]\!]\|_{0,e}^2
     \lesssim  & (f,\varphi_e)_{\omega_e} +\sum_{K\in \mcal T_e} (\nabla^2(v_h-w),\nabla^2\varphi_e)_{K}
       \notag\\
       &+ \iota^2\sum_{K\in \mcal T_e}(\nabla^3(v_h-w),\nabla^3\varphi_e)_{K} - \iota^2\sum_{K\in \mcal T_e}(\nabla^3v_h ,\nabla^3\varphi_e)_{K} .
  \end{align}
  Since $\varphi_e \in H^3_0(\omega_e)$, we get from the similar arguments for deriving \eqref{eq:discreteFormIntegration5} that
     \begin{align}\label{eq:GEKnnvh1}
     \iota^2(\nabla^3v_h ,\nabla^3\varphi_e)_{\omega_e} =
       \iota^2 \sum_{K\in \mcal T_e}   (\partial_{\bm{nnn}}v_h,\partial_{\bm{nn}}\varphi_e) _{\partial K} =     \iota^2    ([\![ \partial_{\bm{nnn}}v_h ]\!],\partial_{\bm{nn}}\varphi_e)_e.
  \end{align}
  By \eqref{eq:GEKnnvh}-\eqref{eq:GEKnnvh1} and the inverse inequality, it holds
  \begin{align*}%\label{eq:GEKnnvh2}
    \| [\![  \partial_{\bm{nn}} v_h ]\!]\|_{0,e}^2\lesssim&
       \|f\|_{0,\omega_e}\|\varphi_e\|_{0,\omega_e}+\sum_{K\in \mcal T_e}|v_h-w|_{2,K}|\varphi_e|_{2,K}\notag\\
       &+ \iota^2\sum_{K\in \mcal T_e}|v_h-w|_{3,K}|\varphi_e|_{3,K}
        +\iota^2  \|[\![ \partial_{\bm{nnn}}v_h ]\!]\|_{0,e} \|\partial_{\bm{nn}}\varphi_e\|_{0,e}
       \notag\\
       \lesssim&  (   \|f\|_{0,\omega_e} +\iota^2  h_e^{-5/2} \|[\![ \partial_{\bm{nnn}}v_h ]\!]\|_{0,e})\|\varphi_e\|_{0,\omega_e}
          \notag\\
    & +  \Big(h_e^{-4}\sum_{K\in \mcal T_e}|v_h-w|_{2,K}^2+\iota^4 h_e^{-6}\sum_{K\in \mcal T_e}|v_h-w|_{3,K}^2 \Big)^{1/2}\|\varphi_e\|_{0,\omega_e} ,
  \end{align*}
  which together with \eqref{eq:GEKerrorf_Deltavh}, \eqref{eq:GEKerrorVhnnn} and \eqref{eq:GEKvarphie} yields \eqref{eq:GEKerrorVhnn} immediately.
\end{proof}

\begin{lemma}\label{lm:GEKStrangErrorah}
For any $w_I, v_h\in V_h$, it holds
  \begin{align*}
  \iota^2a_h(w_I, v_h- E_h^c v_h) \lesssim&(Osc_h(f) +  \|w- w_I\|_{\iota,h})\|v_h\|_{\iota,h},
\end{align*}
 where
  \begin{equation}\label{eq:GEKOSCf}
    Osc_h(f)^2 = \sum_{K\in \mcal T_h}h_K^4\|f-\Pi_K^0f\|_{0,K}^2.
  \end{equation}
\end{lemma}
\begin{proof}
  %{\color{red} If \eqref{eq:GEKJump} is false, then we can construct a weak interpolation operator $I_h^c:H^3_0\to W_h\cap C^2(\overline{\Omega})$ by the $L^2$ projection operator and the definition of marco-element. Then we can obtain the similar error estimate in Lemma \ref{lm:GEKStrangInterpolation}.
%  \begin{align}
%    \|[\![\partial_{\bm{nt}} w_I  ]\!] \|_{0,e}\lesssim& h_e^{-1}\|[\![\partial_{\bm{n}} w_I  ]\!] \|_{0,e}
%     = h_e^{-1}\|[\![\partial_{\bm{n}} (w_I-I_h^c w)  ]\!] \|_{0,e}
%     \notag\\
%   \leq & h_e^{-1}\|[\![\partial_{\bm{n}} (w_I- w)  ]\!] \|_{0,e}+ h_e^{-1}\|[\![\partial_{\bm{n}} (w-I_h^c w)  ]\!] \|_{0,e}
%  \end{align}
%    }
 Let $\beta_{v_h} =  v_h- E_h^c v_h$. Employing the same arguments for proving \eqref{eq:discreteFormIntegration5}, one gets
 \begin{align*}%\label{GEKStrangErrorah5}
  \sum_{K\in \mcal T_h}   (\nabla^3 w_I,\nabla^3\beta_{v_h})_K
      = &    \sum_{K\in \mcal T_h}  (\partial_{\bm{nnn}} w_I, \partial_{\bm{nn}}\beta_{v_h})_{\partial K} +2  \sum_{K\in \mcal T_h}    (\partial_{\bm{nnt}} w_I, \partial_{\bm{nt}}\beta_{v_h}  )_{\partial K}
      \notag\\
      &
       +      \sum_{K\in \mcal T_h}  (\partial_{\bm{ntt}} w_I, \partial_{\bm{tt}}\beta_{v_h} )_{\partial K}    .
  \end{align*}
  %Since $w, E_h^cv_h\in H^3_0(\Omega)$, it's easy to know that
%  \begin{equation}\label{eq:GEKJump}
%   [\![\partial_{\bm{nn}}w ]\!]|_e =[\![\partial_{\bm{n}}w ]\!] |_e =[\![\partial_{\bm{nn}} E_h^cv_h ]\!]|_e =[\![\partial_{\bm{n}} E_h^cv_h ]\!] |_e=0\quad \forall\;e\in \mcal E(\mcal T_h).
%  \end{equation}
Then, recalling the definition of the bilinear form $a_h(\cdot,\cdot)$, we get %from \eqref{GEKStrangErrorah5} that
  \begin{align}\label{eq:GEKStrangErrorah2}
     &a_h(w_I, \beta_{v_h})
      \notag\\
      =&  \sum_{e \in  \mcal{E}^i(\mcal T_h)}
      ([\![\partial_{\bm{nnn}} w_I]\!], \{\!\{\partial_{\bm{nn}}\beta_{v_h}\}\!\})_e
      + 2 \sum_{e \in  \mcal{E}^i(\mcal T_h)}
      ([\![\partial_{\bm{nnt}} w_I]\!], \{\!\{\partial_{\bm{nt}}\beta_{v_h}\}\!\})_e
      \notag\\
      &+\sum_{e \in  \mcal{E}^i(\mcal T_h)}
       ([\![\partial_{\bm{ntt}} w_I]\!], \{\!\{\partial_{\bm{tt}}\beta_{v_h}\}\!\})_e
        - \sum_{e \in \mcal{E}(\mcal T_h)}
      ([\![\partial_{\bm{nn}}w_I]\!], \{\!\{\partial_{\bm{nnn}}\beta_{v_h}\}\!\})_e
         \notag\\
     & - 2\sum_{e \in  \mcal{E}(\mcal T_h)}
       ([\![\partial_{\bm{nt}} w_I  ]\!] , \{\!\{\partial_{\bm{nnt}}\beta_{v_h}\}\!\})_e
       + \eta\sum_{e \in  \mcal{E}(\mcal T_h)}  h_e^{-1}( [\![  \partial_{\bm{nn}}w_I]\!],[\![  \partial_{\bm{nn}} \beta_{v_h} ]\!])_e
               \notag\\
    &  + \eta\sum_{e \in  \mcal{E}(\mcal T_h)}  h_e^{-3}( [\![  \partial_{\bm{n}}  w_I  ]\!], [\![  \partial_{\bm{n}} \beta_{v_h}  ]\!])_e  .
  \end{align}
  Combining the Cauchy-Schwarz inequality, \eqref{eq:GEKerrorVhnnn} and \eqref{eq:GEKenrichingError1}-\eqref{eq:GEKenrichingError2}, we find
  \begin{align}\label{eq:GEKStrangErrorah3}
  &  \iota^2\sum_{e \in  \mcal{E}^i(\mcal T_h)}   ([\![\partial_{\bm{nnn}} w_I]\!], \{\!\{\partial_{\bm{nn}}\beta_{v_h}\}\!\})_e
         \notag \\
         \lesssim &\sum_{e \in  \mcal{E}^i(\mcal T_h)} \Big(  h_e^{4}\|f-\Pi_K^0f\|_{0,\omega_e}^2  + \sum_{K\in \mcal T_e} |w-v_h|_{2,\omega_e}^2\Big)^{1/2}h_e^{1/2}\| \{\!\{\partial_{\bm{nn}}\beta_{v_h}\}\!\}\|_{0,e}
         \notag\\
      &+\sum_{e \in  \mcal{E}^i(\mcal T_h)} \iota^2h_e^{-1/2}\Big(\sum_{K\in \mcal T_e}|w-v_h|_{3,K}^2\Big)^{1/2}
       \| \{\!\{\partial_{\bm{nn}}\beta_{v_h}\}\!\}\|_{0,e}
         \notag\\
    \lesssim & ( Osc_h(f)  + |w-v_h|_{2,h})\interleave  v_h \interleave_{2,h}   +\iota^2 |w-v_h|_{3,h}  \interleave  v_h \interleave_{3,h}.
  \end{align}
 By the Cauchy-Schwarz inequality and the inverse inequality, it holds
    \begin{align}\label{eq:GEKStrangErrorah4}
 & 2\iota^2\sum_{e \in  \mcal{E}^i(\mcal T_h)}([\![\partial_{\bm{nnt}} w_I]\!], \{\!\{\partial_{\bm{nt}}  \beta_{v_h}\}\!\})_e  - 2 \iota^2\sum_{e \in  \mcal{E}(\mcal T_h)} ([\![\partial_{\bm{nt}} w_I ]\!] , \{\!\{\partial_{\bm{nnt}}  \beta_{v_h}\}\!\})_e
 \notag\\
      \lesssim&  \iota^2\sum_{e \in  \mcal{E}(\mcal T_h)}   \Big(h_e^{-2}\|[\![\partial_{\bm{nn}} w_I]\!]\|_{0,e}   \| \{\!\{\partial_{\bm{n}}  \beta_{v_h}\}\!\} \|_{0,e}+h_e^{-2}\|[\![\partial_{\bm{n}}w_I ]\!]\|_{0,e} \| \{\!\{\partial_{\bm{nn}}  \beta_{v_h}\}\!\}\|_{0,e}\Big)
 \notag\\
     =&  \iota^2\sum_{e \in  \mcal{E}(\mcal T_h)}   h_e^{-2}\|[\![\partial_{\bm{nn}} (w_I-w)]\!]\|_{0,e}   \|   \{\!\{ \partial_{\bm{n}}  \beta_{v_h}\}\!\}\|_{0,e}
     \notag\\
     &+ \iota^2\sum_{e \in  \mcal{E}(\mcal T_h)} h_e^{-2}\|[\![\partial_{\bm{n}}(w_I-w) ]\!]\|_{0,e} \| \{\!\{\partial_{\bm{nn}}  \beta_{v_h}\}\!\}\|_{0,e} .
  \end{align}
  Here we have used the fact that
  \begin{equation}\label{eq:GEKJump1}
   [\![\partial_{\bm{nn}}w ]\!]|_e =[\![\partial_{\bm{n}}w ]\!] |_e =0\quad \forall\;e\in \mcal E(\mcal T_h),
  \end{equation}
 which also implies
  \begin{align}
     &- \iota^2\sum_{e \in \mcal{E}(\mcal T_h)}
      ([\![\partial_{\bm{nn}}w_I]\!], \{\!\{\partial_{\bm{nnn}}\beta_{v_h}\}\!\})_e\notag
      \\
      \lesssim & \iota^2 \sum_{e \in \mcal{E}(\mcal T_h)}\|[\![\partial_{\bm{nn}}(w-w_I)]\!]\|_{0,e} \|  \{\!\{\partial_{\bm{nnn}}\beta_{v_h}\}\!\}\|_{0,e}.
  \end{align}
Similarly, we get from \eqref{eq:GEKJump1}, the inverse inequality and \eqref{eq:GEKStranginverseP0} that
  \begin{align}\label{eq:GEKStrangErrorah6}
      \iota^2 \sum_{e \in  \mcal{E}^i(\mcal T_h)}
       ([\![\partial_{\bm{ntt}} w_I]\!], \{\!\{\partial_{\bm{tt}}\beta_{v_h}\}\!\})_e
    \lesssim&   \iota^2\sum_{e \in  \mcal{E}^i(\mcal T_h)}   h_e^{-4}\|[\![\partial_{\bm{n}}  w_I  ]\!]\|_{0,e} \| \{\!\{\beta_{v_h}\}\!\} \|_{0,e}
         \notag\\
   \lesssim & \iota^2\sum_{e \in  \mcal{E}^i(\mcal T_h)}   h_e^{-9/2}\|[\![\partial_{\bm{n}}(w_I-w) ]\!]\|_{0,e} \| \beta_{v_h}\|_{0,\omega_e} .
  \end{align}
Using the Cauchy-Schwarz inequality, \eqref{eq:GEKJump} and \eqref{eq:GEKJump1}, one has
  \begin{align}\label{eq:GEKStrangErrorah5}
        &\eta\sum_{e \in  \mcal{E}(\mcal T_h)}  h_e^{-1}( [\![  \partial_{\bm{nn}} w_I ]\!],[\![  \partial_{\bm{nn}} \beta_{v_h}  ]\!])_e  + \eta\sum_{e \in  \mcal{E}(\mcal T_h)}  h_e^{-3}( [\![  \partial_{\bm{n}}   w_I   ]\!], [\![  \partial_{\bm{n}} \beta_{v_h}   ]\!])_e
               \notag\\
= &\eta\sum_{e \in  \mcal{E}(\mcal T_h)}  h_e^{-1}( [\![  \partial_{\bm{nn}}( w_I -w)]\!],[\![  \partial_{\bm{nn}} v_h ]\!])_e  + \eta\sum_{e \in  \mcal{E}(\mcal T_h)}  h_e^{-3}( [\![  \partial_{\bm{n}}(  w_I   -w) ]\!], [\![  \partial_{\bm{n}} v_h  ]\!])_e
    \notag\\
    \lesssim&\sum_{e \in  \mcal{E}(\mcal T_h)}  h_e^{-1}\| [\![  \partial_{\bm{nn}}( w_I -w)]\!]\|_{0,e}\|[\![  \partial_{\bm{nn}} v_h ]\!]\|_{0,e}
  \notag  \\
    &+ \sum_{e \in  \mcal{E}(\mcal T_h)}  h_e^{-3} \| [\![  \partial_{\bm{n}}(  w_I   -w) ]\!]\|_{0,e}\| [\![  \partial_{\bm{n}} v_h  ]\!] \|_{0,e}
    %\lesssim& \|w-w_I\|_{\iota,h}\|v_h\|_{\iota,h}
    .
  \end{align}
Finally, we end the proof by combining \eqref{eq:GEKStrangErrorah2}-\eqref{eq:GEKStrangErrorah5} and \eqref{eq:GEKenrichingError2}.
%    \begin{align}
%   a_h(w_I, v_h- E_h v_h)
%         \lesssim &( Osc_h(f)  +(1+\eta) |w-v_h|_{\iota,h}) \interleave  v_h \interleave_{\iota,h}
%         \notag\\
%          &+\iota|w-v_h|_{\iota,h} \Big(  \sum_{e \in \mcal{E}(\mcal T_h)}h_e^{-6}\| \epsilon \|_{0,\omega_e}^2 +  h_e \| \partial_{\bm{nnn}}\epsilon \|_{0,e}^2+ h_e^{-1}\| \partial_{\bm{nn}}\epsilon \|_{0,e}^2 \Big)^{1/2}
%          \notag\\
%          \lesssim &( Osc_h(f)  + (1+\eta)|w-v_h|_{\iota,h}) \interleave  v_h \interleave_{\iota,h}.
%  \end{align}
\end{proof}

\begin{lemma}\label{lm:fbherror}
 For any $w_I, v_h\in V_h$, it holds
  \begin{align*}
   (f, v_h - E_h^c v_h) - b_h(w_I,v_h - E_h^c v_h)\lesssim&(Osc_h(f) +  \|w- w_I\|_{\iota,h})\|v_h\|_{\iota,h},
  \end{align*}
  where, $Osc_h(f)$ is determined by \eqref{eq:GEKOSCf}.
\end{lemma}
\begin{proof}
Using the similar arguments for deriving \eqref{eq:discreteFormIntegration11}, we have
 \begin{align*}
   \sum_{K\in \mcal T_h}(\nabla^2w_I, \nabla^2\beta_{v_h})_K
  =& %\sum_{K\in \mcal T_h} (\Delta^2 w_I, \beta_{v_h}) _K+
   \sum_{K\in \mcal T_h}(\partial_{\bm{nn}} w_I, \partial_{\bm{n}} \beta_{v_h})_{\partial K}
   + \sum_{K\in \mcal T_h}(\partial_{\bm{nt}} w_I, \partial_{\bm{t}} \beta_{v_h})_{\partial K}
  \notag\\
  =& \sum_{e\in\mcal E^i(\mcal T_h)}([\![\partial_{\bm{nn}} w_I]\!],      \{\!\{ \partial_{\bm{n}} \beta_{v_h}\}\!\})_e
  + \sum_{e\in\mcal E(\mcal T_h)} ( \{\!\{ \partial_{\bm{nn}} w_I\}\!\}, [\![\partial_{\bm{n}}\beta_{v_h}]\!])_e
  \notag\\
  &  + \sum_{e\in\mcal E^i(\mcal T_h)}([\![\partial_{\bm{nt}} w_I]\!],      \{\!\{ \partial_{\bm{t}} \beta_{v_h}\}\!\})_e.
\end{align*}
Then, it holds
    \begin{align}\label{eq:fbherror3}
        (f, \beta_{v_h}) - b_h(w_I,\beta_{v_h})
    = &   (f, \beta_{v_h}) - \sum_{e\in\mcal E^i(\mcal T_h)}([\![\partial_{\bm{nn}} w_I]\!],    \{\!\{ \partial_{\bm{n}} \beta_{v_h}\}\!\})_e
    \notag\\
    &  -\sum_{e\in\mcal E^i(\mcal T_h)}([\![\partial_{\bm{nt}} w_I]\!],      \{\!\{ \partial_{\bm{t}} \beta_{v_h}\}\!\})_e  + \sum_{e\in\mcal E(\mcal T_h)}([\![\partial_{\bm{n}} w_I]\!],    \{\!\{ \partial_{\bm{nn}}\beta_{v_h}\}\!\})_e
    \notag\\
    &
    - \eta\sum_{e \in \mcal{E}(\mcal T_h)} h_e^{-1}( [\![  \partial_{\bm{n}} w_I ]\!], [\![  \partial_{\bm{n}} \beta_{v_h} ]\!])_e .
  \end{align}
  Applying the Cauchy-Schwarz inequality and \eqref{eq:GEKerrorf_Deltavh}, we get
     \begin{align*}
      (f,\beta_{v_h})
       \lesssim & \sum_{K\in \mcal T_h}\|f\|_{0,K}\|\beta_{v_h}\|_{0,K}
      \notag\\
      \lesssim& \sum_{K\in \mcal T_h}( h_K^{2} \|f-\Pi_K^0f\|_{0,K} +|w-v_h|_{2,K}) h_K^{-2}\|\beta_{v_h}\|_{0,K}
      \notag\\
      & +  \sum_{K\in \mcal T_h}\iota^2h_K^{-3}|w-v_h|_{3,K}\|\beta_{v_h}\|_{0,K},
  \end{align*}
which together with \eqref{eq:GEKenrichingError1} and \eqref{eq:GEKenrichingError2} yields
 \begin{align}\label{eq:fbherror4}
  %    (f, v_h - E_h v_h)
%      \lesssim&\sum_{K\in \mcal T_h}( h_K^{2} \|f-\Pi_K^0f\|_{0,K} +|w-v_h|_{2,K}+ \iota^2h_K^{-1}|w-v_h|_{3,K}) \interleave  v_h \interleave_{2,h}
%      \notag\\
         (f, \beta_{v_h})
      \lesssim& ( Osc_h(f) + |w- w_I|_{2,h}) \interleave  v_h \interleave_{2,h}
        +   \iota^2|w- w_I|_{3,h}\interleave  v_h \interleave_{3,h}
      \notag\\
      \lesssim &(Osc_h(f) + \|w- w_I\|_{\iota,h})\|v_h\|_{\iota,h}.
  \end{align}
Similarly, by \eqref{eq:GEKerrorVhnn}, \eqref{eq:GEKenrichingError1} and \eqref{eq:GEKenrichingError2}, 
      \begin{align}\label{eq:fbherror5}
      &  - \sum_{e\in\mcal E^i(\mcal T_h)}([\![\partial_{\bm{nn}} w_I]\!],
      \{\!\{ \partial_{\bm{n}}\beta_{v_h}\}\!\}   )_e
       \notag\\
       \lesssim &  \sum_{e\in\mcal E^i(\mcal T_h)}\|[\![\partial_{\bm{nn}} w_I]\!]\|_{0,e}\| \{\!\{ \partial_{\bm{n}} \beta_{v_h}\}\!\}\|_{0,e}
      \notag\\
      \lesssim&   \sum_{e\in\mcal E(\mcal T_h)}\sum_{K\in \mcal T_e}( h_K^{4}\|f-\Pi_K^0f\|_{0,K}^2   +|w- w_I|_{2,K}^2)^{1/2}h_e^{-1/2} \| \{\!\{ \partial_{\bm{n}}\beta_{v_h}\}\!\}\|_{0,e}
      \notag\\
      &+\sum_{e\in\mcal E^i(\mcal T_h)}\iota^2h_e^{-3/2}(\sum_{K\in \mcal T_e}|w- w_I|_{3,K}^2)^{1/2}  \| \{\!\{ \partial_{\bm{n}}\beta_{v_h}\}\!\}\|_{0,e}
      \notag\\
         \lesssim &(Osc_h(f) + \|w- w_I\|_{\iota,h})\|v_h\|_{\iota,h}.
  \end{align}
Similarly, we get from \eqref{eq:GEKJump1}, the inverse inequality, \eqref{eq:GEKStranginverseP0} and \eqref{eq:GEKenrichingError1} that
  \begin{align}
& - \sum_{e \in  \mcal{E}^i(\mcal T_h)}
       ([\![\partial_{\bm{nt}} w_I]\!], \{\!\{\partial_{\bm{t}}\beta_{v_h}\}\!\})_e\\
    \lesssim&   \sum_{e \in  \mcal{E}(\mcal T_h)}   h_e^{-2}\|[\![\partial_{\bm{n}}  w_I  ]\!]\|_{0,e} \| \{\!\{\beta_{v_h}\}\!\} \|_{0,e}
         \notag\\
   \lesssim &\big( \sum_{e \in  \mcal{E}^i(\mcal T_h)}   h_e^{-1}\|[\![\partial_{\bm{n}}(w_I-w) ]\!]\|_{0,e}^2\big) ^{1/2} \big(\sum_{e \in  \mcal{E}^i(\mcal T_h)}h_e^{-4}\| \beta_{v_h}\|_{0,\omega_e}^2\big) ^{1/2}
   \notag\\
   \lesssim & \interleave  w_I-w\interleave_{2,h}  \interleave  v_h \interleave_{2,h} .
  \end{align}
  By \eqref{eq:GEKJump}, \eqref{eq:GEKJump1} and \eqref{eq:GEKenrichingError1}, we have
 \begin{align}\label{eq:fbherror6}
  & \sum_{e \in \mcal{E}(\mcal T_h)} ([\![\partial_{\bm{n}} w_I]\!],    \{\!\{ \partial_{\bm{nn}} \beta_{v_h}\}\!\})_e
 \notag\\
   = &\sum_{e \in \mcal{E}(\mcal T_h)} ([\![\partial_{\bm{n}} (w_I-w)]\!],    \{\!\{ \partial_{\bm{nn}} \beta_{v_h}\}\!\})_e
    \notag\\
\lesssim &\Big(\sum_{e \in \mcal{E}(\mcal T_h)}h_e^{-1}\|[\![\partial_{\bm{n}} (w_I-w)]\!]\|_{0,e}^2\Big)^{1/2}
    \Big(\sum_{e \in \mcal{E}(\mcal T_h)}h_e\|\{\!\{ \partial_{\bm{nn}}   \beta_{v_h}\}\!\}\|_{0,e}^2\Big)^{1/2}
\notag\\
\lesssim &  \interleave  w_I-w\interleave_{2,h}  \interleave  v_h \interleave_{2,h}
  \end{align}
  and
  \begin{align}\label{eq:fbherror7}
- \eta\sum_{e \in \mcal{E}(\mcal T_h)} h_e^{-1}( [\![  \partial_{\bm{n}} w_I ]\!], [\![  \partial_{\bm{n}}    \beta_{v_h} ]\!])_e =&- \eta\sum_{e \in \mcal{E}(\mcal T_h)} h_e^{-1}( [\![  \partial_{\bm{n}}( w_I-w) ]\!], [\![  \partial_{\bm{n}} v_h  ]\!])_e
  \notag\\
\lesssim &  \interleave  w_I-w\interleave_{2,h}  \interleave  v_h \interleave_{2,h} .
  \end{align}
   Hence, we obtain the desired result by \eqref{eq:fbherror3}-\eqref{eq:fbherror7}.
\end{proof}

   \begin{lemma}\label{lm:GEKHermiteResidual}
    Let $w\in V$ be the solution of problem \eqref{weakForm:GEK}. Then for any $w_I,v_h\in V_h$, it holds
       \begin{align}
      \iota^2a_h( w- w_I ,E_h^c  v_h) + b_h (w- w_I,E_h^c v_h)
       \lesssim & \|   w- w_I\|_{\iota,h} \|  v_h\|_{\iota,h}.\label{eq:GEKHermiteResidual}
       \end{align}
     \end{lemma}
    \begin{proof}
 For simplicity, let $\beta_w = w-w_I$. By \eqref{eq:GEKJump}, we can rewrite the discrete bilinear form as
    \begin{align}\label{eq:GEKHermiteResidual1}
  &\iota^2 a_h(\beta_w,E_h^c v_h) +   b_h(\beta_w ,E_h^c v_h)
  \notag\\
  = &\iota^2\sum_{K\in \mcal T_h}(\nabla^3 \beta_w, \nabla^3 E_h^c v)_K- \iota^2\sum_{e \in  \mcal{E}(\mcal T_h)}([\![\partial_{\bm{nn}}\beta_w ]\!], \{\!\{\partial_{\bm{nnn}} E_h^c v\}\!\})_e
  \notag\\
    &
               - 2\iota^2\sum_{e \in  \mcal{E}(\mcal T_h)}([\![\partial_{\bm{nt}}\beta_w ]\!] , \{\!\{\partial_{\bm{nnt}} E_h^c v\}\!\})_e
               \notag \\
        &+  \sum_{K\in \mcal T_h}(\nabla^2 \beta_w, \nabla^2 E_h^c v)_K  - \sum_{e \in  \mcal{E}(\mcal T_h)} (  [\![  \partial_{\bm{n}} \beta_w ]\!] , \{\!\{ \partial_{\bm{nn}} E_h^c v \}\!\})_e.
  \end{align}
  By \eqref{eq:GEKStranginverseP0} and the Cauchy-Schawarz inequality, one has
      \begin{align}\label{eq:GEKHermiteResidual2}
 & - \iota^2\sum_{e \in  \mcal{E}(\mcal T_h)}([\![\partial_{\bm{nn}}\beta_w ]\!], \{\!\{\partial_{\bm{nnn}} E_h^c v\}\!\})_e - \sum_{e \in  \mcal{E}(\mcal T_h)} (  [\![  \partial_{\bm{n}} \beta_w ]\!] , \{\!\{ \partial_{\bm{nn}} E_h^c v \}\!\})_e
  \notag\\
  \lesssim &  \iota^2\sum_{e \in  \mcal{E}(\mcal T_h)}h_e^{-1/2}\|[\![\partial_{\bm{nn}}\beta_w ]\!]\|_{0,e} |  E_h^cv |_{3,\omega_e}
           +\sum_{e \in  \mcal{E}(\mcal T_h)}h_e^{-1/2}\| [\![  \partial_{\bm{n}} \beta_w ]\!]\|_{0,e} |  E_h^c v |_{2,\omega_e}
           \notag\\
        \lesssim & \|\beta_w\|_{\iota,h}  \|E_h^c v \|_{\iota,h}.
  \end{align}
 Using the Cauchy-Schwarz inequality again, we have from \eqref{eq:GEKJump1} and the inverse inequality that
\begin{align}\label{eq:GEKHermiteResidual3}
 &- 2\iota^2\sum_{e \in  \mcal{E}(\mcal T_h)}([\![\partial_{\bm{nt}}\beta_w   ]\!] , \{\!\{\partial_{\bm{nnt}} E_h^c v\}\!\})_e
\notag\\
 =&
    2\iota^2\sum_{e \in  \mcal{E}(\mcal T_h)}([\![\partial_{\bm{nt}}w_I  ]\!] , \{\!\{\partial_{\bm{nnt}} E_h^c v\}\!\})_e
   \notag\\
   \lesssim &\iota^2\sum_{e \in  \mcal{E}(\mcal T_h)}h_e^{-2}\|[\![\partial_{\bm{n}}w_I  ]\!] \|_{0,e} \| \{\!\{\partial_{\bm{nn}} E_h^c v\}\!\}\|_{0,e}
   \notag\\
  \lesssim &\iota^2\big(\sum_{e \in  \mcal{E}(\mcal T_h)}h_e^{-3}\|[\![\partial_{\bm{n}}\beta_w  ]\!] \|_{0,e}^2\big)^{1/2}\big(\sum_{e \in  \mcal{E}(\mcal T_h)}h_e^{-1} \| \partial_{\bm{nn}} E_h^c v \|_{0,e}^2\big)^{1/2}
   \notag\\
   \lesssim& \|\beta_w\|_{\iota,h} \|  E_h^c v \|_{\iota,h}.
\end{align}
The inequalities \eqref{eq:GEKHermiteResidual1}-\eqref{eq:GEKHermiteResidual3} implies
  \begin{align*}
 \iota^2 a_h(\beta_w,E_h^c v_h) +   b_h(\beta_w ,E_h^c v_h)
        \lesssim & \|w-w_I\|_{\iota,h}  \|E_h^c v \|_{\iota,h}.
  \end{align*}
  Applying \eqref{eq:GEKenrichingError3} to the last inequality leads to the desired result immediately.
    \end{proof}

%
%
%´ËÍâ, ÓÖÓÉ~\eqref{eq:GEKenrichingError1}, \eqref{eq:GEKInterpolationError}~ºÍ¼£²»µÈʽ, ¿ÉµÃ
%  \begin{align}\label{eq:fbherror21}
%  &\sum_{e \in \mcal{E}(\mcal T_h)} ([\![\partial_{\bm{n}} w_I]\!],    \{\!\{ \partial_{\bm{nn}} (v_h - E_h v_h)\}\!\})_e
%  \notag\\
%   = &\sum_{e \in \mcal{E}(\mcal T_h)} ([\![\partial_{\bm{n}} ((w_I-I_hw_0) -(w-w_0)+ (I_hw_0-w_0))]\!],    \{\!\{ \partial_{\bm{nn}} (v_h - E_h v_h)\}\!\})_e
%   \notag\\
%\lesssim &\sum_{e \in \mcal{E}(\mcal T_h)} \|[\![\partial_{\bm{n}} ((w_I-I_hw_0) -(w-w_0))]\!]\|_{0,e}\|\{\!\{ \partial_{\bm{nn}} (v_h - E_h v_h)\}\!\}\|_{0,e}
%\notag\\
%&+\sum_{e \in \mcal{E}(\mcal T_h)}\|[\![\partial_{\bm{n}} (I_hw_0-w_0)]\!]\|_{0,e} \|\{\!\{ \partial_{\bm{nn}} (v_h - E_h v_h)\}\!\}\|_{0,e}
%\notag\\
%\lesssim &( |w-w_0|_2 + h^2|w_0|_{4} ) \interleave  v_h \interleave_{2,h}.
%  \end{align}
%  ÔÙʹÓÃͬÑùµÄ¼¼ÇÉÓÐ
%   \begin{align}\label{eq:fbherror22}
%  &- \eta\sum_{e \in \mcal{E}(\mcal T_h)} h_e^{-1}( [\![  \partial_{\bm{n}} w_I ]\!], [\![  \partial_{\bm{n}}  (v_h - E_h v_h) ]\!])_e
%   \notag\\
%\lesssim &\eta\sum_{e \in \mcal{E}(\mcal T_h)}h_e^{-1}\|[\![\partial_{\bm{n}}(w_I-I_hw_0) -(w-w_0)]\!]\|_{0,e}\|[\![ \partial_{\bm{n}} (v_h - E_h v_h)]\!]\|_{0,e}
%\notag\\
%&+\eta\sum_{e \in \mcal{E}(\mcal T_h)}\|[\![\partial_{\bm{n}} (I_hw_0-w_0)]\!]\|_{0,e}\|[\![ \partial_{\bm{n}} (v_h - E_h v_h)]\!]\|_{0,e}
%\notag\\
%\lesssim &\eta( |w-w_0|_2 + h^2|w_0|_{4} ) \interleave  v_h \interleave_{2,h}.
%  \end{align}
\begin{lemma}
   Let $w\in V$ and $w_h\in V_h$ be the solutions of problems \eqref{weakForm:GEK} and \eqref{eq:discreteStrangForm}, respectively. Then it holds
          \begin{align}
            \|w-w_h\|_{\iota,h} \lesssim& Osc_h(f) + (1+\eta)\inf\limits_{w_I\in V_h}\|w- w_I\|_{\iota,h} , \label{eq:GEKStrangerrorCea}
          \end{align}
            where, $Osc_h(f)$ is determined by \eqref{eq:GEKOSCf}.
\end{lemma}
\begin{proof}
   For any $ w_I\in V_h$, let $\tilde{w}_h= w_h- w_I$. By Lemma \ref{lm:GEKStrangstability}, \eqref{eq:discreteStrangForm} and \eqref{weakForm:GEK}, it follows that
        \begin{align*}%\label{eq:GEKStrangerrorOrder3}
          \| w_h- w_I\|_{\iota,h}^2
         \lesssim&  (f,\tilde{w}_h ) -\iota^2a_h( w_I ,\tilde{w}_h)  - b_h(  w_I ,\tilde{w}_h)
          \notag\\
         = &  (f,\tilde{w}_h -E_h^c\tilde{w}_h) - b_h (w_I,\tilde{w}_h-E_h^c\tilde{w}_h)- \iota^2a_h( w_I , \tilde{w}_h-E_h^c\tilde{w}_h)
          \notag\\
                          &+ \iota^2a_h( w- w_I , E_h^c\tilde{w}_h) + b_h (w- w_I ,E_h^c\tilde{w}_h).
        \end{align*}
  Then, applying Lemma \ref{lm:GEKStrangErrorah}, Lemma \ref{lm:fbherror} and Lemma \ref{lm:GEKHermiteResidual} with $v_h = \tilde{w}_h$ leads to
    \begin{align}\label{eq:GEKStrangerrorOrder4}
     \| w_h- w_I\|_{\iota,h}^2
         \lesssim & \big(Osc_h(f) + (1+\eta)\|w- w_I\|_{\iota,h}\big) \| w_h- w_I\|_{\iota,h} .
        \end{align}
        The proof is completed by the triangle inequality and \eqref{eq:GEKStrangerrorOrder4}.
\end{proof}

\begin{theorem}
   Let $w\in V\cap H^4(\Omega)$, $w_0\in H^2_0(\Omega)\cap H^s(\Omega)$ and $w_h\in V_h$ be the solutions of problems \eqref{weakForm:GEK}, \eqref{weakForm:KPB} and \eqref{eq:discreteStrangForm}, respectively. Then the following error estimates hold:
          \begin{align}
            \|w-w_h\|_{\iota,h} \lesssim&Osc_h(f) + (h^2+\iota h) |w|_4,  \label{eq:GEKStrangerrorOrder3}  \\
            \|w-w_h\|_{\iota,h} \lesssim& Osc_h(f) +\iota^{1/2}\|f\|_0 + h^{s-2} |w_0|_s,  \label{eq:GEKStrangerrorOrder1}%  \\
           % \|w_0-w_h\|_{\iota,h}\lesssim &  (\iota^{1/2}+ h^{-1}\iota^{3/2}  )\|f\|_{0}  + h^{s-2} |w_0|_s    ,\label{eq:GEKStrangerrorOrder2}
          \end{align}
          where $s=3,4$ and $Osc_h(f)$ is determined by \eqref{eq:GEKOSCf}.
\end{theorem}
\begin{proof}
  Setting $w_I$ as $ I_h w$ in \eqref{eq:GEKStrangerrorCea}. Then the inequality \eqref{eq:GEKStrangerrorOrder3} is derived by \eqref{eq:GEKInterEw} and the inequality \eqref{eq:GEKStrangerrorOrder1} is obtained by \eqref{eq:GEKInterEf}.
\end{proof}
\begin{theorem}\label{th:errorw0w}
   Let $w_0\in H^2_0(\Omega)\cap H^s(\Omega)$ and $w_h\in V_h$ be the solutions of problems \eqref{weakForm:KPB} and \eqref{eq:discreteStrangForm}, respectively. Then it holds
          \begin{align*}
          %  \|w_0-w_h\|_{\iota,h}\lesssim & Osc_h(f) + (\iota^{1/2}+ h^{-1}\iota^{3/2}  )\|f\|_{0}  + h^{s-2} |w_0|_s
           |w_0-w_h |_{2,h}  +\iota |w_0-w_h |_{3,h} \lesssim & Osc_h(f) + \iota^{1/2} \|f\|_{0}  + h^{s-2} |w_0|_s ,%\label{eq:GEKStrangerrorOrder2}
          \end{align*}
           where $s=3,4$ and $Osc_h(f)$ is determined by \eqref{eq:GEKOSCf}.
\end{theorem}
\begin{proof}
  Using \eqref{eq:GEKRegularity} and \eqref{eq:regularityKPB}, one has
   \begin{align}\label{eq:05271}
   |w_0-w |_{2,h}  +\iota |w_0-w |_{3,h}
     \lesssim \iota^{1/2}\|f\|_0.
   \end{align}
  % Then, by the trace inequality and \eqref{eq:GEKRegularity}, it holds
% By the trace inequality, \eqref{eq:regularityKPB} and \eqref{eq:GEKRegularity}, it holds
%\begin{align}\label{eq:GEKStrangerrorOrder6}
%    \| w-w_0 \|_{\iota,h}^2
%    = & |w-w_0 |_{2}^2+\iota^2 |  w-w_0 |_3^2  + \iota^2 \sum\limits_{e\in \mcal E(\mcal T_h)}h_e^{-1}\|[\![ \partial_{\bm{nn} }(w-w_0 )]\!] \|_{0,e}^2
%    \notag\\
%    \lesssim & (1 + h^{-2}\iota^2 )|w-w_0 |_{2}^2+\iota^2 |  w-w_0 |_3^2
%    \notag\\
%    \lesssim & (\iota+ h^{-2}\iota^3  )\|f\|_{0}^2   .
%    \end{align}
Applying the triangle inequality and \eqref{eq:GEKStrangerrorOrder1} we obtain the inequality as required.
%\begin{align*}%\label{eq:GEKerrorOrder6}
%    \|  w_0 -w_h\|_{\iota,h}
%    \lesssim &  \|  w_0 -w\|_{\iota,h}   +\|  w -w_h\|_{\iota,h}
%    \notag\\
%    \lesssim &Osc_h(f)+  \iota^{1/2} \|f\|_{0}+ h^{s-2} |w_0|_s  .
%    \end{align*}
%as required.
\end{proof}

%\subsection{A medius error estimate}
 Lastly, we are going to demonstrate that the method \eqref{eq:discreteStrangForm} is convergent without any additional regularity assumption for the exact solution $w$.
\begin{theorem}
    Let $w\in V$ and $w_h\in V_h$ be the solutions of problems \eqref{weakForm:GEK} and \eqref{eq:discreteStrangForm}, respectively. Then we have
    \begin{equation}
      \|w-w_h\|_{\iota,h}\to 0,\quad h\to 0.
    \end{equation}
\end{theorem}
\begin{proof}
As is known to all, $ C^{\infty}(\overline{\Omega})\cap H^3(\Omega)$ is dense in $H^3(\Omega)$. Thus, for any $\epsilon>0$, there exists $\tilde{w}\in C^{\infty}(\overline{\Omega})$ such that $\|w-\tilde{w}\|_3< \epsilon$ and it holds
\begin{align}
   \|w-\tilde{w}\|_{\iota,h}^2 = & |w-\tilde{w}|_{2}^2 + \iota^2 |w-\tilde{w}|_{3}^2
   \notag\\
    < & (1 +\iota^2 )\epsilon^2.
  % \notag\\
   %\leq & (1+\iota )\|w-\tilde{w}\|_3
\end{align}
  Notice that $Osc(f)=O(h^2)$. Setting $w_I$ as $I_h \tilde{w}$ in \eqref{eq:GEKStrangerrorCea} and employing the similar arguments for deriving \eqref{eq:GEKInterEw}, one finds
  \begin{align}
%     \|w-\Pi_{\Omega}^3 \tilde{w}\|_{\iota,h}\leq  \|w-\tilde{w}\|_{\iota,h} + \| \tilde{w}- \Pi_{\Omega}^3  \tilde{w}\|_{\iota,h}
      \|w-w_h\|_{\iota,h}\lesssim& (1+\eta)\|w-I_h \tilde{w}\|_{\iota,h} + O(h^2)
      \notag\\
     \lesssim&  (1+\eta)(\|w-\tilde{w}\|_{\iota,h} + \| \tilde{w}- I_h \tilde{w}\|_{\iota,h}) + O(h^2)
      \notag\\
     \lesssim &(1+\eta)((1+\iota)\epsilon +  (h+\iota )h |\tilde{w}|_4 )+ O(h^2),
  \end{align}
Hence, it holds
  \begin{equation}
    \limsup\limits_{h\to 0} \|w-w_h\|_{\iota,h} \lesssim \epsilon,
  \end{equation}
  and the proof is completed by the arbitrariness of $\epsilon$ and $ \|w-w_h\|_{\iota,h}\geq0$.
\end{proof}

\section{Numerical experiments}
\begin{table}[t]%\color{red}
	\small
	\centering
	\caption{\small The error $\|w-w_h\|_{\iota,h}$ for Example \ref{ex:GEK1} with $\eta=10$.}
	%\begin{threeparttable}
	\begin{tabular}{cccccccccc}
		\toprule
 	\multirow{2}{*}{$\iota$}  & \multicolumn{5}{c}{$h$}  &\multirow{2}{*}{rate} \\
		\cline{2-6}
	 & 1/4 & 1/8 & 1/16 & 1/32& 1/64 \\
		\midrule[0.8pt]
        $1$ &  9.133e+01 & 5.226e+01 & 2.699e+01 &  1.361e+01 & 6.818e+00  & 1.00\\
        %\specialrule{0em}{1pt}{1pt}
       $10^{-2}$& 1.090e+01& 5.996e+00& 2.696e+00& 1.326e+00 &6.729e-01 & 0.98\\
       % \specialrule{0em}{1pt}{1pt}
       $ 10^{-4}$& 4.513e+00 & 1.583e+00 & 4.929e-01 & 1.711e-01 &7.269e-02 &1.24\\
      %  \specialrule{0em}{1pt}{1pt}
        $10^{-6}$& 4.398e+00& 1.446e+00 &3.736e-01 & 8.883e-02 &2.224e-02 & 2.00\\
      %  \specialrule{0em}{1pt}{1pt}
        $0$& 4.397e+00& 1.444e+00 & 3.722e-01& 8.761e-02&2.113e-02& 2.05 \\
		\bottomrule
	\end{tabular}
	\label{Tab:GEK1}
\end{table}

\begin{table}[t]%\color{red}
	\small
	\centering
	\caption{\small The performance for Example \ref{ex:GEK1}.}
	%\begin{threeparttable}
	\begin{tabular}{ccccccccccc}
		\toprule
		\multirow{2}{*}{$\iota$}&	\multirow{2}{*}{$\eta$}  & \multicolumn{5}{c}{$h$}  &\multirow{2}{*}{rate} \\
		\cline{3-7}
	 &\multirow{2}{*}{} & 1/4 & 1/8 & 1/16 & 1/32& 1/64 \\
		\midrule[0.8pt]
%		\multirow{4}{*}{$10^{-1}$}
%      & $10^{-4}$&2.441e+02& 8.554e+01 & 9.483e+01& 2.455e+01 & 4.736e+01&-\\
%        \specialrule{0em}{1pt}{1pt}
%        &$1$ & 1.135e+02& 2.950e+01 &   2.288e+01&2.012e+01&1.897e+01&0.09\\
%        \specialrule{0em}{1pt}{1pt}
%      & $ 10^{4}$& 4.323e+01&4.201e+01& 3.775e+01 & 2.697e+01&1.307e+01&1.04\\
%        \specialrule{0em}{1pt}{1pt}
%      &  $10^{6}$&4.349e+01& 4.348e+01 & 4.343e+01&4.322e+01&  4.241e+01 &0.03\\
%		\midrule[0.8pt]
		\multirow{5}{*}{$10^{-8}$}
      & $10^{-4}$& 4.200e+00&1.231e+00 &3.203e-01&7.985e-02 & 1.989e-02& 2.01\\
        \specialrule{0em}{1pt}{1pt}
 &$10^{-6}$&4.200e+00  &1.231e+00   & 3.203e-01 &  7.985e-02 & 1.989e-02 & 2.01\\
        \specialrule{0em}{1pt}{1pt}
     &  $1$ &4.230e+00& 1.266e+00& 3.376e-01& 8.538e-02&    2.137e-02 &2.00 \\
        \specialrule{0em}{1pt}{1pt}
     &  $ 10^{4}$& 7.303e+00 & 3.523e+00&1.663e+00  &7.591e-01   & 3.132e-01& 1.28\\
        \specialrule{0em}{1pt}{1pt}
     &   $10^{6}$& 7.323e+00  &3.565e+00  & 1.731e+00&8.575e-01&4.259e-01 & 1.01\\
		\bottomrule
	\end{tabular}
	\label{Tab:GEK17}
\end{table}

\begin{table}[tpb]%\color{red}
	\small
	\centering
	\caption{\small The performance for Example \ref{ex:GEK2} with $\iota=10^{-6}$ and $\eta=10$.}
	%\begin{threeparttable}
	\begin{tabular}{ccccccccccc}
		\toprule
		 $h$  & $ 1/4$&  $ 1/8$&$ 1/16$  & $ 1/32$  & $ 1/64$ \\
		 \midrule[1pt]
      $ \|w_0- w_h\|_{\iota,h}  $ & 2.941e+00 & 8.607e-01 & 2.098e-01 & 4.997e-02 & 1.256e-02   \\
        rate&-&  1.77 &2.04 &2.07&   1.99
  \\
  $|w_0- w_h|_{1}$&   1.806e-01 &2.513e-02  &2.958e-03& 3.457e-04& 4.206e-05 \\
        rate&- &  2.84&3.09&3.10& 3.04\\
	    $|w_0- w_h|_{2,h}$& 2.840e+00& 8.113e-01& 1.888e-01 &4.286e-02& 1.027e-02\\
      rate&- & 1.81& 2.10& 2.14& 2.06\\
       $|w_0- w_h|_{3,h}$  &  5.073e+01 & 2.837e+01 & 1.381e+01&6.609e+00   & 3.243e+00\\
    rate&-&0.84& 1.04&1.06& 1.03\\
       $\|w_0- w_h\|_0$& 2.723e-02 & 3.063e-03& 2.545e-04& 1.760e-05& 1.138e-06\\
      rate&- & 3.15& 3.59& 3.85& 3.95\\
		\bottomrule
	\end{tabular}
	\label{Tab:GEK2}
\end{table}
\begin{table}[tpb]%\color{red}
	\small
	\centering
	\caption{\small The performance for Example \ref{ex:GEK2} with $\iota=10^{-8}$ and $\eta=10$.}
	%\begin{threeparttable}
	\begin{tabular}{ccccccccccc}
		\toprule
		 $h$  & $ 1/4$&  $ 1/8$&$ 1/16$  & $ 1/32$  & $ 1/64$ \\
		 \midrule[1pt]
      $ \|w_0- w_h\|_{\iota,h}  $ & 2.940e+00   &8.600e-01 & 2.092e-01&4.943e-02 &1.207e-02  \\
       rate&-& 1.77& 2.04& 2.08&2.03
  \\
  $|w_0- w_h|_{1}$& 1.806e-01& 2.513e-02& 2.958e-03 &  3.457e-04&4.206e-05  \\
        rate&- &  2.84&3.09& 3.10&3.04\\
	    $|w_0- w_h|_{2,h}$& 2.840e+00 &  8.113e-01 &1.888e-01& 4.286e-02 & 1.027e-02  \\
          rate&- &  1.81 &2.10&2.14  & 2.06\\
       $|w_0- w_h|_{3,h}$  &  5.073e+01& 2.837e+01 & 1.381e+01 & 6.609e+00&3.243e+00  \\
      rate&-&0.84& 1.04&1.06&1.03\\
       $\|w_0- w_h\|_0$& 2.723e-02 &3.063e-03 &  2.545e-04 &1.760e-05 &1.136e-06 \\
     rate&- &  3.15& 3.59&3.85&3.95\\
		\bottomrule
	\end{tabular}
	\label{Tab:GEK22}
\end{table}
%\begin{table}[t]%\color{red}
%	\small
%	\centering
%	\caption{\small The performance for Example \ref{ex:GEK2} with $\iota=1e-8$ and $h=1/64$.}
%	%\begin{threeparttable}
%	\begin{tabular}{ccccccccccc}
%		\toprule
%		 $\eta$  &$\|w_0- w_h\|_{\iota,h}$  & $|w_0- w_h|_{3,h}$ & $|w_0-w_h|_{2,h}$ & $\|w_0- w_h\|_0$ \\
%		 \midrule[1pt]
%        $1$& 1.230e-02&3.291e+00& 9.883e-03& 5.142e-07  \\
%%\specialrule{0em}{1pt}{1pt}
%        $10^{2}$&1.686e-02 &4.260e+00&1.580e-02& 8.366e-06 \\
%%\specialrule{0em}{1pt}{1pt}
%	    $10^{4}$&  2.474e-01&5.428e+01&2.473e-01&2.135e-04\\
%%\specialrule{0em}{1pt}{1pt}
% $10^6 $& 3.739e-01&8.349e+01&3.737e-01& 5.317e-04  \\
%%\specialrule{0em}{1pt}{1pt}
%        $0$& 1.132e-02&3.227e+00&9.062e-03&4.115e-07  \\
%		\bottomrule
%	\end{tabular}
%	\label{Tab:GEK21}
%\end{table}

\begin{example}\label{ex:GEK1}\rm
	%This example to testify two parameter robustness with respect to $\lambda$ and $\iota$.
The exact solution of \eqref{model:GEK} is chosen as
\begin{align*}
       w=& \sin^3(\pi x)\sin^3(\pi y),
	\end{align*}
 and the right side function is determined by \eqref{model:GEK} and $w$. The example is used to verify the estimate \eqref{eq:GEKStrangerrorOrder3}. Let $\eta=10$. We display the error $ \|w-w_h \|_{\iota,h}$ in Tab. \ref{Tab:GEK1} with different mesh size $h$ and size parameter $\iota$. Besides, we test the robustness of the numerical method with respect to the penalization parameter $\eta$ in Tab. \ref{Tab:GEK17}.
 \end{example}

 As shown in Tab. \ref{Tab:GEK1}, we observe that the convergence rate of $\|w- w_h\|_{\iota,h}$ for $\iota= 1,10^{-2},10^{-4}$ is close to $O(h)$, while one for $\iota= 0,10^{-6}$ would reach to $O(h^2)$ nearly. Since $Osc_h(f) =O(h^2)$, it can be shown that the numerical results are in coincidence with the theoretical predication in \eqref{eq:GEKStrangerrorOrder3}. Moreover, when $\iota=10^{-8}$, it can been seen from Tab. \ref{Tab:GEK17} that the numerical method \eqref{eq:discreteStrangForm} behaves well for a large range of $\eta$.
%when $\iota=10^{-1}$, we can see from Tab. \ref{Tab:GEK21} that the method \eqref{eq:discreteStrangForm} may not be convergent for $\eta=10^{-4}$. And it does't has the optimal convergence rate for $\eta=10^6,1$. However,

\begin{example}\label{ex:GEK2}\rm
	%This example to testify two parameter robustness with respect to $\lambda$ and $\iota$.
The exact solution of the reduced problem \eqref{model:KPB} is set to be a function in the form
\begin{align*}
       w_0=& \sin^2(\pi x)\sin^2(\pi y).
	\end{align*}
Set $f$ computed from \eqref{model:KPB} be the right side term of \eqref{model:GEK}. Then, the exact solution of \eqref{model:GEK} possesses strong boundary layer. Besides, It's easy to check that $ f$ is independent of $\iota$ and the exact solution of \eqref{model:GEK} is unknown.   We use this example to demonstrate the robustness of the proposed method \eqref{eq:discreteStrangForm} with respect to size parameter $\iota$. %In particular, the optimal convergence rate can be achieved when $\iota \to 0$.
We display the error $ w_0-w_h $ in Tab. \ref{Tab:GEK2} and Tab. \ref{Tab:GEK22} when $\iota=10^{-6},10^{-8} $ and $\eta =10$. %Moreover, we report the error $w_0-w_h$ in various norm with a fixed mesh size $h=1/64$ and different values of the penalization parameter.
\end{example}
%From \eqref{eq:GEKStrangerrorOrder6}, it holds
%\begin{align}\label{eq:GEK2}
%    \|w-w_h\|_{\iota,h} \lesssim  & \|w_0-w \|_{\iota,h}+    \|w_0-w_h\|_{\iota,h}
%    \notag\\
%    \lesssim & Osc_h(f) + \iota^{1/2}\|f\|_{0}  + h^{s-2} |w_0|_s+    \|w_0-w_h\|_{\iota,h}.
%\end{align}
By \eqref{eq:05271} and $Osc_h(f) =O(h^2)$ , it suffices to observe the behavior of error $
  w_0-w_h $ to verify \eqref{eq:GEKStrangerrorOrder1}. We can observe from Tab. \ref{Tab:GEK2} and Tab. \ref{Tab:GEK22} that when $\iota=10^{-6}$, $\|w_0 - w_h\|_{\iota,h}=O(h^{1.99})$ and when $\iota=10^{-8}$,  $\|w_0 - w_h\|_{\iota,h}=O(h^{2.03})$. It can be seen that the numerical results are in agreement with Theorem \ref{th:errorw0w}, which implies that the proposed method \eqref{eq:discreteStrangForm} is robust with respect to the size parameter $\iota$. Besides, it can be seen that $|w_0 - w_h|_{3,h}=O(h^{1.03})$, $|w_0 - w_h|_{2,h}=O(h^{1.99})$, $|w_0 - w_h|_1=O(h^{3.04})$, $\|w_0 - w_h\|_0=O(h^{3.95})$. %when $\iota=0$, $\|w_0 - w_h\|_{\iota,h}=O(h^{2.03})$, which implies the method also works for approximating the fourth-order elliptic problem. Furthermore, as shown in Tab. \ref{Tab:GEK21}, the method \eqref{eq:discreteStrangForm} works for a large range of $\eta$.

\section{Declarations}
\begin{enumerate}
\item {\bf Conflicts of interest/Competing interests.} The authors have no known competing financial interests or personal relationships that could have appeared to influence the work reported in this manuscript.

\item {\bf Availability of data and material.} The datasets generated during and/or analysed during the current study are available from the corresponding author on reasonable request.

\item {\bf Code availability.}  The codes during the current study are available from the corresponding author on reasonable request.

\item {\bf Authors' contributions.} All authors contributed equally to this manuscript.
\end{enumerate}

\section*{Reference}
\bibliographystyle{abbrv}
\bibliography{GEK-Hermite}

\begin{thebibliography}{10}

\bibitem{Aifantis1984microstructural}
E.~C. Aifantis.
\newblock On the microstructural origin of certain inelastic models.
\newblock {\em J. Eng. Mater. Technol.}, 106(4):326--330, 1984.

\bibitem{AltanAifantis1992}
S.~B. Altan and E.~C. Aifantis.
\newblock On the structure of the mode {III} crack-tip in gradient elasticity.
\newblock {\em Scr. Metall. Mater.}, 26(2):319--324, 1992.

\bibitem{AnHuangZhang2024}
Q.~An, X.~Huang, and C.~Zhang.
\newblock A decoupled finite element method for the triharmonic equation.
\newblock {\em Appl. Math. Lett.}, 147:108843, 8, 2024.

\bibitem{AntoniettiManzini2020}
P.~F. Antonietti, G.~Manzini, and M.~Verani.
\newblock The conforming virtual element method for polyharmonic problems.
\newblock {\em Comput. Math. Appl.}, 79(7):2021--2034, 2020.

\bibitem{Arnold1982}
D.~N. Arnold.
\newblock An interior penalty finite element method with discontinuous
  elements.
\newblock {\em SIAM. J. Numer. Anal.}, 19(4):742--760, 1982.

\bibitem{AshooriMahmoodi2013}
A.~{Ashoori Movassagh} and M.~Mahmoodi.
\newblock A micro-scale modeling of {K}irchhoff plate based on modified
  strain-gradient elasticity theory.
\newblock {\em Eur. J. Mech. A-Solids}, 40:50--59, 2013.

\bibitem{BabuPatel2019}
B.~Babu and B.~Patel.
\newblock A new computationally efficient finite element formulation for
  nanoplates using second-order strain gradient {K}irchhoff's plate theory.
\newblock {\em Composites Part B: Engineering}, 168:302--311, 2019.

\bibitem{BrambleZlamal1970}
J.~H. Bramble and M.~Zl\'{a}mal.
\newblock Triangular elements in the finite element method.
\newblock {\em Math. Comp.}, 24:809--820, 1970.

\bibitem{BrennerNeilan2011}
S.~C. Brenner and M.~Neilan.
\newblock A {$C^0$} interior penalty method for a fourth order elliptic
  singular perturbation problem.
\newblock {\em SIAM J. Numer. Anal.}, 49(2):869--892, 2011.

\bibitem{BrennerSung2005IPFourth}
S.~C. Brenner and L.-Y. Sung.
\newblock {$C^0$} interior penalty methods for fourth order elliptic boundary
  value problems on polygonal domains.
\newblock {\em J. Sci. Comput.}, 22/23:83--118, 2005.

\bibitem{ChenHuangWei2022}
C.~Chen, X.~Huang, and H.~Wei.
\newblock {$H^m$}-conforming virtual elements in arbitrary dimension.
\newblock {\em SIAM J. Numer. Anal.}, 60(6):3099--3123, 2022.

\bibitem{ChenLiQiu2022}
H.~Chen, J.~Li, and W.~Qiu.
\newblock A {$C^0$} interior penalty method for {$m$}th-{L}aplace equation.
\newblock {\em ESAIM Math. Model. Numer. Anal.}, 56(6):2081--2103, 2022.

\bibitem{ChenHuang2020}
L.~Chen and X.~Huang.
\newblock Nonconforming virtual element method for {$2m$}th order partial
  differential equations in {$\Bbb{R}^n$}.
\newblock {\em Math. Comp.}, 89(324):1711--1744, 2020.

\bibitem{ChenHuang2021Decomposition}
L.~Chen and X.~Huang.
\newblock Geometric decompositions of the simplicial lattice and smooth finite
  elements in arbitrary dimension, 2021, arXiv:2111.10712.

\bibitem{ChenHuangHuang2023}
M.~Chen, J.~Huang, and X.~Huang.
\newblock A robust lower-order mixed finite element method for a strain
  gradient elastic model.
\newblock {\em SIAM J. Numer. Anal.}, 61(5):2237--2260, 2023.

\bibitem{Cosserat1909}
E.~Cosserat and F.~Cosserat.
\newblock Theorie des corp deformables.
\newblock {\em Herman, Paris}, 1909.

\bibitem{DassiMoraReales2024}
F.~Dassi, D.~Mora, C.~Reales, and I.~Vel\'asquez.
\newblock Mixed variational formulations of virtual elements for the
  polyharmonic operator {$(-\Delta)^n$}.
\newblock {\em Comput. Math. Appl.}, 158:150--166, 2024.

\bibitem{FuZhou2020}
G.~Fu, S.~Zhou, and L.~Qi.
\newblock On the strain gradient elasticity theory for isotropic materials.
\newblock {\em Internat. J. Engrg. Sci.}, 154:103348, 24, 2020.

\bibitem{Gallistl2017}
D.~Gallistl.
\newblock Stable splitting of polyharmonic operators by generalized {S}tokes
  systems.
\newblock {\em Math. Comp.}, 86(308):2555--2577, 2017.

\bibitem{Grisvard1985}
P.~Grisvard.
\newblock {\em Elliptic Problems in Nonsmooth Domains}.
\newblock Pitman, Boston, MA, 1985.

\bibitem{Gudi2010medius}
T.~Gudi.
\newblock A new error analysis for discontinuous finite element methods for
  linear elliptic problems.
\newblock {\em Math. Comp.}, 79(272):2169--2189, 2010.

\bibitem{Gudi2010DG}
T.~Gudi.
\newblock Some nonstandard error analysis of discontinuous {G}alerkin methods
  for elliptic problems.
\newblock {\em Calcolo}, 47(4):239--261, 2010.

\bibitem{Gudi2011sixthIPDG}
T.~Gudi and M.~Neilan.
\newblock An interior penalty method for a sixth-order elliptic equation.
\newblock {\em IMA J. Numer. Anal.}, 31(4):1734--1753, 2011.

\bibitem{Guzman2012Nitsche}
J.~Guzm\'{a}n, D.~Leykekhman, and M.~Neilan.
\newblock A family of non-conforming elements and the analysis of {N}itsche's
  method for a singularly perturbed fourth order problem.
\newblock {\em Calcolo}, 49(2):95--125, 2012.

\bibitem{Hjelle2006}
O.~y. Hjelle and M.~D\ae~hlen.
\newblock {\em Triangulations and applications}.
\newblock Mathematics and Visualization. Springer-Verlag, Berlin, 2006.

\bibitem{HuLinWu2023}
J.~Hu, T.~Lin, and Q.~Wu.
\newblock A construction of {$C^r$} conforming finite element spaces in any
  dimension.
\newblock {\em Found. Comput. Math.}, 24:1941--1977, 2024.

\bibitem{HuZhang2017}
J.~Hu and S.~Zhang.
\newblock A canonical construction of {$H^m$}-nonconforming triangular finite
  elements.
\newblock {\em Ann. Appl. Math.}, 33(3):266--288, 2017.

\bibitem{HuZhang2019}
J.~Hu and S.~Zhang.
\newblock A cubic {$H^3$}-nonconforming finite element.
\newblock {\em Commun. Appl. Math. Comput.}, 1(1):81--100, 2019.

\bibitem{Huang2020Nonconforming}
X.~Huang.
\newblock Nonconforming virtual element method for $2m$th order partial
  differential equations in $\mathbb{R}^n$ with $m>n$.
\newblock {\em Calcolo}, 57(4):42, 2020.

\bibitem{HuangShiWang2021singular}
X.~Huang, Y.~Shi, and W.~Wang.
\newblock A {Morley-Wang-Xu} element method for a fourth order elliptic
  singular perturbation problem.
\newblock {\em J. Sci. Comput.}, 87(3):84, 2021.

\bibitem{HuangLai2013}
X.~H. Huang, J.~J. Lai, and W.~Q. Wang.
\newblock A modified {A}rgyris element method for {K}irchhoff plates bending
  problems.
\newblock {\em J. Shanghai Jiaotong Univ. (Chinese Ed.)}, 47(2):203--209, 2013.

\bibitem{IshaquddiGopalakrishnan2020}
M.~Ishaquddin and S.~Gopalakrishnan.
\newblock A novel weak form quadrature element for gradient elastic beam
  theories.
\newblock {\em Appl. Math. Model.}, 77:1--16, 2020.

\bibitem{IshaquddinGopalakrishnan2021}
M.~Ishaquddin and S.~Gopalakrishnan.
\newblock Differential quadrature-based solution for non-classical
  {E}uler-{B}ernoulli beam theory.
\newblock {\em Eur. J. Mech. A Solids}, 86:104135, 26, 2021.

\bibitem{Koiter1964couplestress}
W.~T. Koiter.
\newblock Couple-stresses in the theory of elasticity, {I and II}.
\newblock {\em Nederl. Akad. Wetensch. Proc. Ser. B}, 67:17--44, 1964.

\bibitem{LiJi2021}
A.~Li, X.~Ji, S.~Zhou, L.~Wang, J.~Chen, and P.~Liu.
\newblock Nonlinear axisymmetric bending analysis of strain gradient thin
  circular plate.
\newblock {\em Appl. Math. Model.}, 89:363--380, 2021.

\bibitem{LiWu2024}
J.~Li and S.~Wu.
\newblock A construction of canonical nonconforming finite element spaces for
  elliptic equations of any order in any dimension, 2024, arXiv:2409.06134.

\bibitem{Mindlin1964Micro}
R.~D. Mindlin.
\newblock Micro-structure in linear elasticity.
\newblock {\em Arch. Ration. Mech. Anal.}, 16(1):51--78, 1964.

\bibitem{MindlinTiersten1962Effects}
R.~D. Mindlin and H.~F. Tiersten.
\newblock Effects of couple-stresses in linear elasticity.
\newblock {\em Arch. Ration. Mech. Anal.}, 11(1):415--448, 1962.

\bibitem{MousaviNiiranen2015}
S.~M. Mousavi, J.~Niiranen, and A.~H. Niemi.
\newblock Differential cubature method for gradient-elastic {K}irchhoff plates.
\newblock {\em J. Structural Mech.}, 48(3):164--180, 2015.

\bibitem{NiiranenKiendl2017GEK}
J.~Niiranen, J.~Kiendl, A.~H. Niemi, and A.~Reali.
\newblock Isogeometric analysis for sixth-order boundary value problems of
  gradient-elastic {K}irchhoff plates.
\newblock {\em Comput. Methods Appl. Mech. Engrg.}, 316:328--348, 2017.
\newblock Special Issue on Isogeometric Analysis: Progress and Challenges.

\bibitem{NiiranenNiemi2017GEK}
J.~Niiranen and A.~H. Niemi.
\newblock Variational formulations and general boundary conditions for
  sixth-order boundary value problems of gradient-elastic {K}irchhoff plates.
\newblock {\em Eur. J. Mech. A-Solids}, 61:164--179, 2017.

\bibitem{NilssenTaiWinther2001}
T.~Nilssen, X.-C. Tai, and R.~Winther.
\newblock A robust nonconforming {$H^2$}-element.
\newblock {\em Math. Comput.}, 70(234):489--505, 2001.

\bibitem{Nitsche1971}
J.~Nitsche.
\newblock {\"{U}ber ein Variationsprinzip zur L\"{o}sung von
  Dirichlet-Problemen bei Verwendung von Teilr\"{a}umen, die keinen
  Randbedingungen unterworfen sind}.
\newblock {\em Abh. Math. Sem. Univ. Hamburg}, 36:9--15, 1971.

\bibitem{PapargyriGiannakopoulos2010}
S.~Papargyri-Beskou, A.~Giannakopoulos, and D.~Beskos.
\newblock Variational analysis of gradient elastic flexural plates under static
  loading.
\newblock {\em Internat. J. Solids Structures}, 47(20):2755--2766, 2010.

\bibitem{PegiosPapargyri2015}
I.~P. Pegios, S.~Papargyri-Beskou, and D.~E. Besko.
\newblock Finite element static and stability analysis of gradient elastic beam
  structures.
\newblock {\em Acta Mech.}, 226(3):745--768, 2015.

\bibitem{RenLiu2024}
Y.~Ren and D.~Liu.
\newblock The hybrid high-order method for {$p$}-{L}aplace problem.
\newblock {\em Math. Numer. Sin.}, 46(4):397--408, 2024.

\bibitem{RuAifantis1993simple}
C.~Q. Ru and E.~C. Aifantis.
\newblock A simple approach to solve boundary-value problems in gradient
  elasticity.
\newblock {\em Acta Mech.}, 101(1-4):59--68, 1993.

\bibitem{Schedensack2016}
M.~Schedensack.
\newblock A new discretization for {$m$}th-{L}aplace equations with arbitrary
  polynomial degrees.
\newblock {\em SIAM J. Numer. Anal.}, 54(4):2138--2162, 2016.

\bibitem{Shi-Wang-2013}
Z.~Shi and M.~Wang.
\newblock {\em Finite Element Methods}.
\newblock Academic Press, Beijing, 2013.

\bibitem{Toupin1962Elastic}
R.~A. Toupin.
\newblock Elastic materials with couple-stresses.
\newblock {\em Arch. Ration. Mech. Anal.}, 11(1):385--414, 1962.

\bibitem{Zenisek1970C2element}
A.~\v{Z}en\'{\i}\v{s}ek.
\newblock Interpolation polynomials on the triangle.
\newblock {\em Numer. Math.}, 15:283--296, 1970.

\bibitem{WangHuangZhao2016}
B.~Wang, S.~Huang, J.~Zhao, and S.~Zhou.
\newblock Reconsiderations on boundary conditions of {K}irchhoff micro-plate
  model based on a strain gradient elasticity theory.
\newblock {\em Appl. Math. Model.}, 40(15):7303--7317, 2016.

\bibitem{WangXu2013}
M.~Wang and J.~Xu.
\newblock Minimal finite element spaces for {$2m$}-th-order partial
  differential equations in {$R^n$}.
\newblock {\em Math. Comp.}, 82(281):25--43, 2013.

\bibitem{WangHuangTang2018singular}
W.~Wang, X.~Huang, K.~Tang, and R.~Zhou.
\newblock {Morley-Wang-Xu} element methods with penalty for a fourth order
  elliptic singular perturbation problem.
\newblock {\em Adv. Comput. Math.}, 44(4):1041--1061, 2018.

\bibitem{WuXu2019}
S.~Wu and J.~Xu.
\newblock Nonconforming finite element spaces for {$2m$}th order partial
  differential equations on {$\Bbb{R}^n$} simplicial grids when {$m=n+1$}.
\newblock {\em Math. Comp.}, 88(316):531--551, 2019.

\bibitem{XuDengMeng2014}
X.~Xu, Z.~Deng, J.~Meng, and K.~Zhang.
\newblock Bending and vibration analysis of generalized gradient elastic
  plates.
\newblock {\em Acta Mech.}, 225(12):3463--3482, 2014.

\bibitem{YanVescovini2023}
C.~A. Yan, R.~Vescovini, and N.~Fantuzzi.
\newblock A neural network-based approach for bending analysis of strain
  gradient nanoplates.
\newblock {\em Eng. Anal. Bound. Elem.}, 146:517--530, 2023.

\bibitem{ZHOU2023}
Y.~Zhou and K.~Huang.
\newblock On simplified deformation gradient theory of modified gradient
  elastic {Kirchhoff-Love} plate.
\newblock {\em Eur. J. Mech. A Solids}, 100:105014, 2023.

\end{thebibliography}
\end{document}